\documentclass[12pt]{amsart}%

\usepackage{amssymb}
\usepackage{amsmath}
\usepackage{amsthm}
\usepackage{amscd}
\usepackage[margin=1in]{geometry}
\usepackage{hyperref}

\theoremstyle{plain}
\newtheorem{theorem}{Theorem}[section]
\newtheorem{proposition}{Proposition}[section]
\newtheorem{corollary}{Corollary}[section]
\newtheorem{lemma}{Lemma}[section]

\theoremstyle{remark}
\newtheorem{definition}{Definition}[section]
\newtheorem{remark}{Remark}[section]

\DeclareMathOperator{\supp}{supp}

\DeclareMathOperator{\lin}{span}

\DeclareMathOperator{\diag}{diag}
\DeclareMathOperator{\lip}{Lip}

\begin{document}

\title{Sup-norm-closable bilinear forms and Lagrangians}
\author{Michael Hinz$^1$}
\address{Fakult\"at f\"ur Mathematik, Universit\"at Bielefeld, Postfach 100131, 33501 Bielefeld, Germany}
\email{mhinz@math.uni-bielefeld.de}
\thanks{$^1$Research supported in part by NSF grant DMS-0505622 and by the Alexander von Humboldt Foundation (Feodor Lynen Research Fellowship Program)}

\date{\today}

\begin{abstract}
We consider symmetric non-negative definite bilinear forms on algebras of bounded  real valued functions and investigate closability with respect to the supremum norm. In particular, any Dirichlet form gives rise to a sup-norm closable bilinear form. Under mild conditions a sup-norm closable bilinear form admits finitely additive energy measures. If, in addition, there exists a (countably additive) energy dominant measure, then a sup-norm closable bilinear form can be turned into a Dirichlet form admitting a carr\'e du champ. Moreover, we can always transfer the  bilinear form  to an isometrically isomorphic algebra of bounded functions on the Gelfand spectrum, where these measures exist. Our results complement a former closability study of Mokobodzki for the locally compact and separable case. 
 
\tableofcontents
\end{abstract}
\maketitle

\section{Introduction}

The present article is concerned with nonnegative definite symmetric bilinear forms $\mathcal{E}$ on a given space $\mathcal{D}$ of bounded real valued functions on a nonempty set $X$. A simple example is given by the space $\mathcal{D}=\lip(\Omega)$ of Lipschitz functions on a smooth bounded domain $\Omega\in\mathbb{R}^n$ together with the Dirichlet integral
\begin{equation}\label{E:classical}
\mathcal{E}(f,g)=\int_\Omega \nabla f\nabla g\:dx,\ \ f,g\in\lip(\mathbb{R}^n).
\end{equation}
In many cases such bilinear forms $(\mathcal{E},\mathcal{D})$ give rise to \emph{(symmetric) Dirichlet forms}. 
In modern terminology, cf. \cite{BH91, ChF12, FOT94, MR92, S74}, we use this name for  closed nonnegative definite symmetric bilinear forms that are densely defined on Hilbert spaces of $L_2$-equivalence classes of functions on $X$ with respect to a suitable measure (called the \emph{volume measure} or \emph{reference measure}) and have a certain contraction property. 
The theory of Dirichlet forms embeds naturally into the theory of general closed forms on Hilbert spaces, cf. \cite{RS}, which is omnipresent in the mainstream literature in mathematical physics. Because of their physical significance such bilinear forms $\mathcal{E}$ are also referred to as \emph{energy forms}. One reason this name is their relation to the spectral theory of self-adjoint operators on $L_2$-spaces of equivalence classes of functions. Another justification is the energy integral for the potential energy of a vibrating string (or, more generally, the potential energy of the gradient field), which may be formulated in terms of continuous functions and does not need a Hilbert space context.

One motivation for our study comes from the recent need to develop harmonic analysis on spaces that are not locally compact, or have no pre-determined topology (see, for instance,  \cite{BaGor,BaGorMel}). This is related to recent results towards a vector analysis 
on spaces that have no smoothness structure, see \cite{BeMaRe,PeBe,HTa,HTb,HRT}, where it may not always be clear which topology or volume measure is the most appropriate. Another, long term motivation for our study consists in the aim to contribute mathematical tools that can deal with the well established appearance of fractal structures in the physical theory of Quantum Gravity, see, for instance \cite{Englert,Reuter}. Some mathematical models related to Quantum Gravity were recently proposed in \cite{GRV13, GRV14}. A general mathematical theory (yet to be developed) should include a notion of energy that is independent of the underlying space-time metric. In particular, the usual approach to begin a study with a Hilbert space $L^2(X,\mu)$ implicitly fixes a notion of time, which we would like to avoid. In a sense, we use a space $\mathcal{D}$ of bounded finite energy functions as our main object of study, without choosing a particular $L^2$-space. A consequent development of this idea might also allow new insights into the theory of Dirichlet forms on noncommutative spaces, see e.g. \cite{C08, DaLi92, DaRo89}, and related aspects of Lipschitz seminorms, cf. \cite{Rieffel}. In this sense we aim to collect some results that together form a proxy to the rich and well developed theory of Dirichlet forms, \cite{BH91, ChF12, FOT94, MR92}.

In their seminal papers \cite{BD59a, BD59b} Beurling and Deny introduced what we today refer to as \emph{transient extended Dirichlet space}, cf. \cite[p. 41]{FOT94}. Their definition involves a fixed reference measure, and if two functions belong to the same energy equivalence class, they also belong to the same equivalence class in terms of local integrability with respect to this measure. Because of the strong potential theoretic blend of these papers this idea is natural, and several subsequent articles follow this point of view, e.g. \cite{D65, Doob62, M72, M73}. A more abstract framework containing this construction is the functional completion of Aronszajn and Smith \cite{AS56}. Fukushima \cite{F69, F71, F71b} (see also \cite[Appendix A.4]{FOT94}) introduced a more general concept of (not necessarily transient) \emph{Dirichlet spaces}, investigated their representation theory and in the regular case, constructed associated Hunt processes. To adopt this point of view again a reference measure is fixed a priori and $L_2$-spaces of classes of functions are considered. This perspective is related to the theory of commutative von Neumann algebras. It is carefully elaborated in \cite{C06} and generalizes to the noncommutative context, \cite{AHK, C08, DaLi92, DaRo89, S2, S3}. On the other hand, Allain and Andersson,\cite{Allain, And75}, studied the representation theory of bilinear forms on the level of continuous functions. Their results do not involve the choice of a reference measure and rather correspond to the representation theory of general commutative Banach-$^\ast$ (or $C^\ast$-) algebras. Another direction to be mentioned is the theory of resistance forms as developed by Kigami, see for example \cite{Ki01, Ki03, Ki12}. Roughly speaking, they are far reaching generalizations of the classical energy form (\ref{E:classical}) for $n=1$. For resistance forms a comprehensive theory (including Green's operators and Laplacians) can already be developed without fixing a reference measure, cf. \cite{Ki03}.

In the present article we attempt to follow the perspective of \cite{Allain} and \cite{And75}. We start from rather algebraic ingredients, namely an algebra $\mathcal{D}$ of bounded real valued functions (not equivalence classes) on a nonempty set $X$ and a nonnegative definite symmetric bilinear form $\mathcal{E}$ on $\mathcal{D}$ which has a certain contraction property (normal contractions operate). To turn a given bilinear form $(\mathcal{E},\mathcal{D})$ into a Dirichlet form some
$L_2$-density and closability with respect to a suitable measure are needed. In \cite{Mo95} Mokobodzki studied bilinear forms on a dense subspace of the space of continuous functions on a locally compact and second countable base space. \cite[Proposition 1]{Mo95} tells that lower semicontinuity with respect to the sup-norm (\emph{sup-norm-lower semicontinuity}) is equivalent to the existence of a bounded (Borel) measure with respect to which the form is $L_2$-closable. At this point not much can be said about this measure. We will follow another standard idea to single out some suitable measures and consider the linear functionals $L_f$, $f\in\mathcal{D}$, given by 
\[L_f(h):=2\mathcal{E}(fh,f)-\mathcal{E}(f^2,h), \ \ h\in\mathcal{D}.\]
In the theory of Dirichlet forms corresponding functionals are positive and display a contraction property, cf. \cite[Proposition I.4.1.1]{BH91}. For the bilinear form in the above example we obtain
\[L_f(h)=\int_\Omega h |\nabla f|^2dx.\] 
We provide sufficient conditions for their positivity in more general cases.  If positive, then the functionals $L_f$ can be represented by finitely additive energy measures $\Gamma(f)$ that generalize $|\nabla f|^2dx$, i.e. we have $L_f(h)=\int_X h\:d\Gamma(f)$ for a suitable notion of integral, cf. \cite{Hew51}. For regular Dirichlet forms the energy measures $\Gamma(f)$ were already studied by Silverstein \cite{S74}, LeJan \cite{LJ78}, Fukushima \cite{FOT94} and others, and in this case they are known to be countably additive (in fact, Radon measures). If there exists a finite (and countably additive) measure $m$ on $X$ is such that all energy measures $\Gamma(f)$ are absolutely continuous with respect to $m$ (in this case  $m$ is called \emph{energy dominant}), then it may be possible to verify the $L_2$-closability of $(\mathcal{E},\mathcal{D})$ with respect to $m$. We discuss a closability condition in terms of the supremum norm and call $(\mathcal{E},\mathcal{D})$ \emph{sup-norm-closable} if for any 
sequence of functions $(f_n)_n\subset\mathcal{D}$ that is $\mathcal{E}$-Cauchy and tends to zero uniformly on $X$, we have $\lim_n\mathcal{E}(f_n)=0$. From the point of view of Banach-$^\ast$ (or $C^\ast$-) algebras the measure-free notion of sup-norm-closability seems natural. It is also remarkable that sup-norm-closability and a contraction property are the essential ingredients to prove the positivity of all the functionals $L_f$, cf. Theorem \ref{T:positive}. Our main result, Theorem \ref{T:clos}, roughly speaking says that if $(\mathcal{E},\mathcal{D})$ is sup-norm-closable and $m$ is a finite energy dominant measure, then $(\mathcal{E},\mathcal{D})$ is closable in $L_2(X,m)$, and its closure is a Dirichlet form admitting a carr\'e du champ in the sense of Bouleau and Hirsch \cite{BH91}, i.e. all energy measures $\Gamma(f)$ have densities with respect to $m$. For the locally compact and separable situation a corresponding statement follows already from Proposition 1 and Th\'eor\`eme 9 in \cite{Mo95} (although a 
complete proof is not given there). In general sup-norm-lower semicontinuity trivially implies sup-norm-closability, and moreover, if there exists a finite energy dominant measure then also the converse implication is true, cf. Section \ref{S:lsc}. In this sense Theorem \ref{T:clos} may be seen as a generalization of Mokobodzki's result and seems to be in good agreement with his Remarques on \cite[p. 412]{Mo95}. Our proof of Theorem \ref{T:clos} is purely analytic and relies on a rather elementary uniform integrability argument together with the Beurling-Deny representation for related bilinear forms on Euclidean spaces (\emph{coordinate bilinear forms}). A preliminary version of this argument was already sketched in a working paper, \cite{HTarXiv}. If a sup-norm-closable form $(\mathcal{E},\mathcal{D})$ has a certain additional separability property and the energy measures are countably additive, then we can \emph{construct} a finite energy dominant measure $m$ by summing up sufficiently many energy measures. For regular Dirichlet forms this is a standard trick, see for instance \cite{Hino08, Hino10} and also \cite{Ku89, T08}. 
In this sense a sup-closable form leads to a Dirichlet form. Conversely, if $(\mathcal{E},\mathcal{F})$ is a Dirichlet form on $L_2(X,\mu)$, where $\mu$ is a $\sigma$-finite measure, and if $\mathcal{B}$ denotes all bounded measurable functions $f$ whose $\mu$-equivalence classes are in $\mathcal{F}$, then we would expect $(\mathcal{E},\mathcal{B})$ to be a sup-norm-closable bilinear form. For finite $\mu$ this is immediate. To capture the general case we study the functionals $L_f$ from a somewhat more abstract point of view, and following \cite{Mosco} (see also \cite{Cap, FL}) we refer to these objects as \emph{Lagrangians}. Again we are led to a sup-norm-closable bilinear form, what shows that sup-norm-closability is a necessary condition for $L_2$-closability, even for unbounded reference measures. Further, if there is a finite energy dominant measure $m$ for $(\mathcal{E},\mathcal{B})$, then we can change measure from $\mu$ to $m$, and the closure of $(\mathcal{E},\mathcal{B})$ in $L_2(X,m)$ is a 
Dirichlet form that admits energy densities. In terms of the associated semigroup (or Markov process) this amounts to a \emph{time change}. For regular Dirichlet forms results on time change follow by probabilistic arguments that are well known, see for instance \cite[Section 6.2]{FOT94}, further references are given in Section \ref{S:clos}. For finite reference measures one may alternatively invoke \cite{Mo95} to obtain this result. Our method provides a change of measure result to a specific kind of measure (energy dominant), but in a more general (not necessarily topological) setup. 

It is well known that for Dirichlet forms the existence of energy densities (i.e. of a \emph{carr\'e du champ}) entails a number of desirable properties, see \cite[Chapter I]{BH91}, including the validity of interesting functional inequalities, \cite{BaEm85, Led}, and in the local case Gaussian short time asymptotics, see \cite{HinoRam}. The results in \cite{Ka12, Ki08} suggest further consequences. We have encountered further advantages of energy dominant reference measures when dealing with (generalized) $L_2$-bundles of $1$-forms in \cite[Section 2]{HRT}, see also \cite{AR89, AR90, Eb99} in this context. In many classical cases the existence of energy densities is immediate, and in a huge part of the existing literature it is just assumed,  \cite{BaEm85, Eb99, Led, Sturm94, Sturm95}. However, for some Dirichlet forms on fractals is atypical, \cite{BBST, Hino03, Hino05}. In these cases changing to an energy dominant reference measure is a way to this establish property in order to obtain further 
results, \cite{Hino08, Hino10, Ka12, Ki08, Ku89, Ku93, T08}. Of course this means changing the volume measure (and therefore the model), but at least from the point of view of stochastic processes this transformation is not unusual.

It may be that there is no energy dominant measure for a bilinear form $(\mathcal{E},\mathcal{D})$, see for instance \cite[Section 6]{HKT}. In this case we can still switch to a compactification of $X$ and transfer the form to a new form $(\hat{\mathcal{E}},\hat{\mathcal{D}})$ for which energy dominant measures exist. In our setup it seems natural to consider the smallest $C^\ast$-algebra containing $\mathcal{D}$ and to switch to its Gelfand spectrum $\Delta$. In \cite{HKT} we have already started to discuss how to transfer a Dirichlet form $(\mathcal{E},\mathcal{F})$ into another Dirichlet form $(\hat{\mathcal{E}},\hat{\mathcal{F}})$ on the Gelfand spectrum $\Delta$ of the closure of $\mathcal{B}$ (in the above notation). There we also had to transfer the given reference measure on $X$ to a measure on the Gelfand spectrum. Here we do not assume the existence of this reference measure on $X$ but construct a reference measure on the Gelfand spectrum. (We would like to remark that unfortunately \cite{HKT} 
contains a mistake, the method to transfer the reference measure to the spectrum may fail in the case of an infinite measure. This mistake can easily be fixed by using the unitisation of the underlying algebra, cf. \cite{Kaniuth}.

We proceed as follows. In Section \ref{S:main} we introduce notation and setup, provide some essential definitions and state our main results. In Section \ref{S:PLF} we show that sup-norm-closability of $(\mathcal{E},\mathcal{D})$ together with a typical contraction property is already sufficient for the linear functionals $L_f$ to be positive, Theorem \ref{T:positive}. Under these conditions they are also bounded and contractive in a similar way. In Section \ref{S:zerokill} we introduce the killing functional of a given bilinear form and discuss the case of zero killing. If all $L_f$ are positive, then a bilinear form with zero killing can essentially be recovered by the operator norms of these functionals. Section \ref{S:Lagrange} contains a more detailed study of Lagrangians, including associated energies, the case of zero killing and sup-norm-closability. Section \ref{S:Dirichlet} is devoted to the proof of our claim that the concept of sup-norm-closable Lagrangians is sufficiently general to capture arbitrary Dirichlet forms on measure spaces, Theorem \ref{T:Dirichletform}. Section \ref{S:energymeas} provides details of the representation of the functionals $L_f$ by finitely additive measures. In particular, we briefly review Hewitts results on finitely additive measures and integration, \cite{Hew51}.
In Section \ref{S:clos} we finally prove Theorems \ref{T:clos}, which tells that under mild conditions any suitably contractive sup-closable bilinear form together with an energy dominant measure produces a Dirichlet form that admits a carr\'e du champ in the sense of \cite{BH91}. As a byproduct we obtain a change of measure result for general (not necessarily regular) Dirichlet forms, Corollary \ref{C:changeofmeasure}. Section \ref{S:Gelfand} is related to the exposition in \cite{HKT}. A sup-norm-closable form can always be transferred into a sup-norm-closable form on the Gelfand spectrum of the closure of $\mathcal{D}$, and for this transferred form an energy dominant measure can be constructed. Finally, we briefly discuss the idea of sup-norm-lower semicontinuity in Section \ref{S:lsc}. Some comments on the proof of another version of Theorem \ref{T:positive} are provided in an appendix.

\section*{Acknowledgements} 

First of all the author thanks Alexander Teplyaev who 
suggested the idea of sup-norm closability, encouraged a detailed study and supported it during many inspiring conversations. Some preliminary results in the direction of this paper were created during a research stay at the University of Connecticut, supported by NSF grant DMS-0505622 and by the Alexander von Humboldt Foundation (Feodor Lynen Research Fellowship Program).

The author also thanks Raffaela Capitanelli, Alexander Grigoryan and Wolfhard Hansen for valuable comments.

\section{Definitions and main results}\label{S:main}

In this section we introduce our setup, provide the necessary definitions and state our main results.

Let $\mathcal{D}$ be a vector space of bounded real valued functions endowed with the supremum norm $\left\|f\right\|_{\sup}=\sup_{x\in X}|f(x)|$. We assume throughout that together with the natural (pointwise) order, $\mathcal{D}$ is a vector lattice. The cone of nonnegative elements of $\mathcal{D}$ is denoted by $\mathcal{D}^+$. We say that $\mathcal{D}$ has the \emph{Stone property} if $f\in\mathcal{D}$ implies $f\wedge 1\in\mathcal{D}$. Consider the functions $T_\alpha:\mathbb{R}\to\mathbb{R}$ by $T_\alpha(x)=x^+\wedge \alpha$ for $\alpha\geq 0$ and $T_\alpha(x)=(-x^-)\vee \alpha$ for $\alpha\leq 0$. To the function $T_1$ we refer as the \emph{unit contraction}. Since $\mathcal{D}$ consists of bounded functions, the Stone property is equivalent to the fact that $\mathcal{D}$ is \emph{stable under the unit contraction}, i.e. $f\in\mathcal{D}$ implies $T_1(f)\in\mathcal{D}^+$. The same is true for any $T_\alpha$ in place of $T_1$. A function $F:\mathbb{R}\to\mathbb{R}$ with $F(0)=0$ and $|F(x)-F(y)|\leq |x-y|$, $x,y\in\mathbb{R}$, is called a \emph{normal contraction}. We say that \emph{$\mathcal{D}$ is stable under normal contractions} if $f\in\mathcal{D}$ implies  $F(f)\in\mathcal{D}$. Stability under normal contractions implies that $\mathcal{D}$ is an algebra, note that $fg=\frac14((f+g)^2-(f-g)^2)$. It also implies both the lattice and the Stone property (i.e. in this case they do not need to be required separately).

Now let $\mathcal{E}:\mathcal{D}\times\mathcal{D}\to\mathbb{R}$ be a nonnegative definite symmetric bilinear form. For simplicity we silently assume these attributes and just write \emph{bilinear form}. We say that \emph{the unit contraction operates on $(\mathcal{E},\mathcal{D})$} if $\mathcal{E}(T_1(f))\leq \mathcal{E}(f)$, $f\in\mathcal{D}$. If the unit contraction operates on $(\mathcal{E},\mathcal{D})$ then this estimate remains true for any $T_\alpha$ in place of $T_1$. We say that \emph{normal contractions operate on $(\mathcal{E},\mathcal{D})$} if $\mathcal{D}$ is stable under normal contractions and
\begin{equation}\label{E:normalcontract}
\mathcal{E}(F(f))\leq \mathcal{E}(f)
\end{equation}
for any $f\in\mathcal{D}$ and any normal contraction $F$. Of course (\ref{E:normalcontract}) implies the same estimate for all $T_\alpha$ (being special normal contractions). 

The next definition introduces a measure independent closability property which is a key notion for our subsequent considerations.

\begin{definition}
We say that a bilinear form $(\mathcal{E},\mathcal{D})$ is \emph{sup-norm-closable} if for any $\mathcal{E}$-Cauchy sequence $(f_n)_n\subset\mathcal{D}$ with $\lim_{n\to\infty}\left\|f_n\right\|_{\sup}=0$ we have $\lim_{n\to\infty}\mathcal{E}(f_n)=0$.
\end{definition}

We are specifically interested in a family of linear functionals associated with a bilinear form $(\mathcal{E},\mathcal{D})$ on an algebra $\mathcal{D}$. Given $f\in\mathcal{D}$, consider the linear functional $L_f:\mathcal{D}\to\mathbb{R}$ given by 
\begin{equation}\label{E:Lagrange1}
L_f(h):=2\mathcal{E}(fh,f)-\mathcal{E}(f^2,h),
\end{equation}
$h\in\mathcal{D}$. We can obtain a bilinear form $L:\mathcal{D}\times\mathcal{D}\to\mathcal{D}'$ by polarization,
\[L_{f,g}:=\frac14\left(L_{f+g}-L_{f-g}\right), \ \ f,g\in\mathcal{D}.\]
Note that
\begin{equation}\label{E:Leibniz}
2\mathcal{E}(fg,h)=L_{f,h}(g)+L_{g,h}(f), \ \ f,g,h\in\mathcal{D}.
\end{equation}
A linear functional $L:\mathcal{D}\to\mathbb{R}$ is said to be \emph{positive} if $h\geq 0$ implies $L(h)\geq 0$. If all $L_f$, $f\in\mathcal{D}$, are positive then bilinearity and Cauchy-Schwarz imply the useful estimate
\begin{equation}\label{E:bilinest}
|L_f(h)^{1/2}-L_g(h)^{1/2}|\leq L_{f-g}(h)^{1/2}
\end{equation}
for any $f,g\in\mathcal{D}$ and $h\in\mathcal{D}^+$. In Section \ref{S:PLF} we verify that, roughly speaking, for sup-norm closable forms $(\mathcal{E},\mathcal{D})$ on which normal contractions operate all functionals $L_f$ are positive. 

The following definition discusses the functionals $L_{f,g}$ from a more abstract point of view. 

\begin{definition}\label{D:Lagrange}
Let $\mathcal{D}$ be an algebra of bounded real valued functions having the Stone property and carrying the supremum norm. 
A bilinear map $L:\mathcal{D}\times\mathcal{D}\to \mathcal{D}'$, $(f,g)\mapsto L_{f,g}$,
is called a \emph{Lagrangian} if it is symmetric (i.e. $L_{g,f}=L_{f,g}$), positive definite (i.e. the functional $L_f:=L_{f,f}$ is positive) and if the unit contraction operates on $(L,\mathcal{D})$, i.e. $L_{T_1(f)}\leq L_f$, $f\in\mathcal{D}$. If in addition $\mathcal{D}$ is stable under normal contractions and if for any normal contraction $F$ we have $L_{F(f)}\leq L_f$, $f\in\mathcal{D}$,
then we say that \emph{normal contractions operate on $(L,\mathcal{D})$}. If there exists a bilinear form $(\mathcal{E},\mathcal{D})$ such that identity (\ref{E:Lagrange1}) holds, 
then we refer to $(L,\mathcal{D})$ as \emph{the Lagrangian generated by $(\mathcal{E},\mathcal{D})$}. 
\end{definition}

Of course there are many prior studies that investigated the functionals $L_f$ in the context of Dirichlet forms, see for instance \cite{BH91, FOT94, HinoRam, LJ78, S74} and in particular \cite{Mosco}, which inspired our nomenclature.

Given a Lagrangian $(L,\mathcal{D})$ we can define a bilinear form $(\mathcal{E}_L,\mathcal{D})$ by polarizing
\[\mathcal{E}_L(f):=\frac12\left\|L_f\right\|_{\mathcal{D}'},\ \ f\in\mathcal{D},\]
see Section \ref{S:Lagrange} for details. We refer to $(\mathcal{E}_L,\mathcal{D})$ as the \emph{energy form associated with $(L,\mathcal{D})$}.

Sup-norm-closability can also be defined for Lagrangians. This is helpful to discuss the connection between Lagrangians and Dirichlet forms.

\begin{definition}
A Lagrangian $(L,\mathcal{D})$ is called \emph{sup-norm-closable} if for any $h\in\mathcal{D}^+$ the bilinear form $(f,g)\mapsto L_{f,g}(h)$ is sup-norm-closable.
\end{definition}

If $\mathbf{1}\in\mathcal{D}$ then obviously $(L,\mathcal{D})$ is sup-norm-closable if and only if its associated energy is. In general, if $(L,\mathcal{D})$ is a sup-norm closable Lagrangian then its associated energy $(\mathcal{E}_L,\mathcal{D})$  is a sup-norm closable bilinear form, see Proposition \ref{P:LtoEL}. Under mild conditions (including an energy separability condition, Definition \ref{D:energysep} below) we can also prove the converse, i.e. that a sup-norm closable bilinear form $(\mathcal{E},\mathcal{D})$ generates a sup-norm closable Lagrangian $(L,\mathcal{D})$ by (\ref{E:Lagrange1}), see Corollary \ref{C:ELtoL} below.

In particular, any Dirichlet form gives rise to a sup-norm-closable Lagrangian and therefore also to a sup-norm-closable bilinear form. In other words, Lagrangians can capture the entire Dirichlet form setup. A proof of the following theorem is given in Section \ref{S:Dirichlet}.

\begin{theorem}\label{T:Dirichletform}
Let $(X,\mathcal{X},\mu)$ be a $\sigma$-finite measure space and $(\mathcal{E},\mathcal{F})$ a Dirichlet form on $L_2(X,\mu)$. Let 
\begin{equation}\label{E:B}
\mathcal{B}:=\left\lbrace f\in b\mathcal{X}: \text{ the $\mu$-class of $f$ is in $\mathcal{F}$}\right\rbrace.
\end{equation}
Then formula (\ref{E:Lagrange1}) defines a sup-norm-closable Lagrangian $(L,\mathcal{B})$ on which normal contractions operate. Moreover, $(\mathcal{E},\mathcal{B})$ is a sup-norm-closable bilinear form on which normal contractions operate.
\end{theorem}

Next, we are interested in a measure representation for Lagrangians. To formulate it we use representation theorems for linear functionals in terms of finitely additive measures proved by Hewitt \cite{Hew51}. The term \emph{finitely additive measure} is used for a set function $\mu$ on an algebra $\mathcal{A}$ with values in $[0,+\infty]$ and such that $\mu(\emptyset)=0$ and $\mu(A\cup B)=\mu(A)+\mu(B)$ for any disjoint $A,B\in\mathcal{A}$.

Let $\alpha(\mathcal{S})$ denote the algebra of subsets of $X$ generated by the collection of sets of form $\left\lbrace f>0\right\rbrace$, $f\in\mathcal{D}$. If $\mathcal{D}$ contains a strictly positive function then for any $f\in\mathcal{D}$ there is a finite and finitely additive measure $\Gamma(f)$ (in the sense of \cite{BRBR, DS,Hew51}) on $\alpha(\mathcal{S})$ such that 
\[L_f(h)=\int_Xhd\Gamma(f),\ \ h\in\mathcal{D},\]
see Theorem \ref{T:intrepadd}. In the context of regular Dirichlet forms the measure $\Gamma(f)$ is just the energy measure of $f$, cf. \cite{FOT94, LJ78, S74}. 

Recall that two finitely additive measures $\mu$ and $\nu$ on $\alpha(\mathcal{S})$, $\nu$ is said to be \emph{absolutely continuous with respect to $\mu$}, $\nu<<\mu$, if given $\varepsilon>0$ there exists some $\delta>0$ such that $\nu(A)<\varepsilon$ for any $A\in\alpha(\mathcal{S})$ with $\mu(A)<\delta$, cf. \cite{BRBR}. 

A set function $\mu$ on an algebra $\mathcal{A}$ with values in $[0,+\infty]$ and such that $\mu(\emptyset)=0$ and $\mu(\bigcup_{i=1}^\infty A_i)=\sum_{i=1}^\infty \mu(A_i)$ for any sequence $A_1,A_2,\dots$ of pairwise disjoint sets $A_i\in\mathcal{A}$ with $\bigcup_{i=1}^\infty A_i\in\mathcal{A}$ will be called a \emph{measure} on $\mathcal{A}$. We use the term 'measure' exclusively for countably additive set functions, in contrast to 'finitely additive measure'. Recall that by Caratheodory's theorem any finite measure $\mu$ on $\mathcal{A}$ extends uniquely to a finite measure on the $\sigma$-algebra $\sigma(\mathcal{A})$ generated by $\mathcal{A}$.

Similarly as in \cite{Hino08, Hino10} we consider the following situation. 

\begin{definition}
A measure $m$ on $\alpha(\mathcal{S})$ is called \emph{energy dominant for $(L,\mathcal{D})$} if all $\Gamma(f)$, $f\in\mathcal{D}$, are absolutely continuous with respect to $m$. Given a bilinear form $(\mathcal{E},\mathcal{D})$, we say that $m$ is \emph{energy dominant for $(\mathcal{E},\mathcal{D})$} if it is energy dominant for the Lagrangian the form generates.
\end{definition}

If $m$ is energy dominant, then automatically all $\Gamma(f)$ will be countably additive on $\alpha(\mathcal{S})$ and therefore extend uniquely to measures on the generated $\sigma$-algebra $\sigma(\mathcal{D})$.

The next of our main results tells that if there exists a finite energy dominant meausure then we can pass from a sup-norm-closable form to a Dirichlet form. By $C_c^1(\mathbb{R}^2)$ we denote the space of compactly supported $C^1$-functions on $\mathbb{R}^2$. We say that $\mathcal{D}$ is $C_c^1(\mathbb{R}^2)$-stable if for any $f,g\in\mathcal{D}$ and $\varphi\in C_c^1(\mathbb{R}^2)$ with $\varphi(0)=0$ we have $\varphi(f,g)\in\mathcal{D}$.

\begin{theorem}\label{T:clos}
Let $\mathcal{D}$ be a $C^1_c(\mathbb{R}^2)$-stable algebra containing a strictly positive function $\chi>0$ and let $(\mathcal{E},\mathcal{D})$ be a sup-norm-closable bilinear form on which normal contractions operate. If $m$ is an energy dominant measure for $(\mathcal{E},\mathcal{D})$, then $(\mathcal{E},\mathcal{D})$ is closable on $L_2(X,m)$, and its closure is a Dirichlet form that admits a carr\'e du champ, i.e. all energy measures $\Gamma(f)$, $f\in\mathcal{D}$, have $m$-integrable densities.
\end{theorem}

The $C^1_c(\mathbb{R}^2)$-stability condition in Theorem \ref{T:clos} can be replaced by an invertibility condition also used in Theorem \ref{T:positive2} below, and if all $L_f$, $f\in\mathcal{D}$, are known to be positive, it can be omitted. 

\begin{remark}
As mentioned in the introduction, a version of this theorem for the locally compact case follows already from the statements in \cite{Mo95}, in particular Proposition 1 and Th\'eor\`eme 9. 
\end{remark}

If $(L,\mathcal{D})$ enjoys a certain separability property, then the existence of an energy dominant measure is merely a question of countable additivity.

\begin{definition}\label{D:energysep}
A Lagrangian $(L,\mathcal{D})$ is called \emph{energy separable} if there exists a countable collection of functions $\left\lbrace f_n\right\rbrace_n\subset \mathcal{D}$ such that all $\Gamma(f)$, $f\in\mathcal{D}$, are absolutely continuous with respect to the finitely additive measure given by
\begin{equation}\label{E:standardsum}
\sum_{n=1}^\infty 2^{-n}\:\frac{\Gamma(f_n)(A)}{1+\Gamma(f_n)(X)},\ \ A\in \alpha(\mathcal{S}).
\end{equation}
A bilinear form $(\mathcal{E},\mathcal{D})$ is called \emph{energy separable} if its Lagrangian is energy separable.
\end{definition}

If $X$ carries a locally compact Hausdorff topology, then countable additivity is a consequence of Riesz' representation theorem, see Remark \ref{R:lcccase} (i) below. Let $C_0(X)$ denote the space of continuous functions on $X$ vanishing at infinity.

\begin{corollary}\label{C:domandclos}
Let $X$ be a locally compact Hausdorff space and let $\mathcal{D}$ be a $C^1_c(\mathbb{R}^2)$-stable dense subalgebra of $C_0(X)$ containing a strictly positive function $\chi>0$. If $(\mathcal{E},\mathcal{D})$ is a sup-norm closable bilinear form on which normal contractions operate and which is energy separable, then (\ref{E:standardsum}) provides a finite energy dominant Radon measure $m$ on $X$ such that $(\mathcal{E},\mathcal{D})$ is closable on $L_2(X,m)$ and its closure is a Dirichlet form admitting a carr\'e du champ.
\end{corollary}

For base spaces $X$ without specified topology or a carrying a non-locally compact topology we can transfer a given bilinear form $(\mathcal{E},\mathcal{D})$ to a bilinear form $(\hat{\mathcal{E}},\hat{\mathcal{D}})$ acting on functions on the \emph{Gelfand spectrum} $\Delta$ of the $C^\ast$-algebra generated by $\mathcal{D}$. Details are provided in Section \ref{S:Gelfand}. Since $\Delta$ is always a locally compact Hausdorff space, Corollary \ref{C:domandclos} applies to the transferred form $(\hat{\mathcal{E}},\hat{\mathcal{D}})$, cf. Theorem \ref{T:closGelfand}. This implies the following results on the sup-norm closability of the original form $(\mathcal{E},\mathcal{D})$.

\begin{corollary}\label{C:ELtoL}
Let $\mathcal{D}$ be a $C^1_c(\mathbb{R}^2)$-stable algebra containing a strictly positive function $\chi>0$ and let $(\mathcal{E},\mathcal{D})$ be an energy separable bilinear form on which normal contractions operate. Then the Lagrangian $(L,\mathcal{D})$ generated by $(\mathcal{E},\mathcal{D})$ is sup-norm-closable if and only if $(\mathcal{E},\mathcal{D})$ is sup-norm-closable.
\end{corollary}

\section{Sup-norm closable bilinear forms and positive linear functionals}\label{S:PLF}

In this section we investigate the positivity of the linear functionals $L_f$ as defined in (\ref{E:Lagrange1}) associated with a given bilinear form $(\mathcal{E},\mathcal{D})$.
Roughly speaking, Theorem \ref{T:positive} below shows that if normal contractions operate then sup-norm-closability implies the positivity of the functionals $L_f$, $f\in\mathcal{D}$, and a contraction property. Given two linear functionals $L, M:\mathcal{D}\to\mathbb{R}$, we write $L\leq M$ if the linear functional $M-L$ positive. 

\begin{theorem}\label{T:positive} 
Let $(\mathcal{E},\mathcal{D})$ be a sup-norm-closable bilinear form on which normal contractions operate. Assume in addition that $\mathcal{D}$ is $C_c^1(\mathbb{R}^2)$-stable. Then for any $f\in\mathcal{D}$ the linear functional $L_f$ is positive and bounded, more precisely,
\begin{equation}\label{E:Lagrangesupbound}
\left\|L_f\right\|_{\mathcal{D}'}\leq 2\mathcal{E}(f),
\end{equation}
and for any normal contraction $F$ we have 
\begin{equation}\label{E:LagrangeMarkov}
L_{F(f)}\leq L_f.
\end{equation}
\end{theorem}

Investing a little more effort we obtain the following version of this theorem, which replaces $C_c^1(\mathbb{R}^2)$-stability assumption on $\mathcal{D}$ by an invertibility condition. Given two bounded real valued functions $f$ and $g$ on $X$ we write $fg^{-1}$ to denote the function $x\mapsto \frac{f(x)}{g(x)}$, seen as an extended real valued function. 

\begin{theorem}\label{T:positive2} 
Let $(\mathcal{E},\mathcal{D})$ be a sup-norm-closable bilinear form on which normal contractions operate. Assume in addition that for any two functions $f,g\in \mathcal{D}$ such that $fg^{-1}$ defines a bounded function on $X$, we have $fg^{-1}\in\mathcal{D}$. Then for any $f\in\mathcal{D}$ the linear functional $L_f$ is positive and the estimates (\ref{E:Lagrangesupbound}) and (\ref{E:LagrangeMarkov}) hold.
\end{theorem}

The remainder of this section is devoted to a proof of Theorem \ref{T:positive}, the auxiliary ingredients needed to verify Theorem \ref{T:positive2} are sketched in the Appendix.

We consider related bilinear forms on Euclidean spaces and employ results of Allain \cite{Allain} and Andersson \cite{And75}. Given $f\in\mathcal{D}$ let $I^f\subset\mathbb{R}$ be a bounded open interval such that $[-\left\|f\right\|_{\sup},\left\|f\right\|_{\sup}]\subset I^f$. Given $f,g\in\mathcal{D}$ we write $I^{f,g}:=I^f\times I^g$ and $I^{f,g}_0:=I^{f,g}\setminus \left\lbrace 0\right\rbrace$ and use the coordinate notation $x=(x_1, x_2)$ for $x\in I^{f,g}$. Let $\mathcal{V}$ denote the algebra of functions generated by the set
\[\left\lbrace F=\Phi(\varphi): \varphi\in C^1_c(\mathbb{R}^2),\ \varphi(0)=0 \text{ and } \Phi\in \lip(\mathbb{R}),\ F(0)=0\right\rbrace\]
and let $C_0(I^{f,g}_0)$ denote the subspace of $C(I^{f,g}_0)$ consisting of all functions that vanish at the boundary of $I_0^{f,g}$ (consisting of the boundary of $I^{f,g}$ and the origin). By restriction and the Stone-Weierstrass theorem $\mathcal{V}$ is a uniformly dense subspace of $C_0(I^{f,g}_0)$. Define a bilinear form $E^{f,g}$ on $\mathcal{V}$ by
\[E^{f,g}(F,G):=\mathcal{E}(F(f,g),G(f,g)),\ \ F,G\in \mathcal{V}.\]
Obviously it is symmetric and nonnegative definite and normal contractions operate.
We use the notation $\diag:=\left\lbrace (x,x): x\in \mathbb{R}^2\right\rbrace$. By \cite[Th\'eor\`{e}me 1]{Allain} there exist a nonnegative Radon measure $\kappa^{f,g}$ on $I_0^{f,g}$, a symmetric  nonnegative Radon measure $J^{f,g}$ on $I^{f,g}_0\times I^{f,g}_0\setminus\diag$ and a bilinear form $N^{f,g}$ such that for any $F\in \mathcal{V}$ we have 
\begin{equation}\label{E:BD2}
E^{f,g}(F)=N^{f,g}(F)+\int\int(F(x)-F(y))^2J^{f,g}(dxdy)+\int F^2d\kappa^{f,g}.
\end{equation}
The form $N^{f,g}$ is strongly local, i.e. $N^{f,g}(F,G)=0$ for any $F,G\in \mathcal{V}$ such that $G$ is constant on a neighborhood of $\supp F$. By \cite[Theorem 2.4]{And75} there exist a uniquely determined symmetric family $\left\lbrace \sigma_{ij}^{f,g}\right\rbrace_{i,j=1,2}$ of Radon measures $\sigma^{f,g}_{ij}$ on $I^{f,g}_0$ such that 
\begin{equation}\label{E:measurerep}
N^{f,g}(F)=\sum_{ij}\int\frac{\partial F}{\partial x_i}\frac{\partial F}{\partial x_j}\:d\sigma_{ij}^{f,g}, \ \ F\in C_c^1(I^{f,g}_0).
\end{equation}
Moreover, for any functions $H_1,H_2\in C_c(I^{f,g}_0)$ the measure $\sum_{ij}H_iH_j\:d\sigma^{f,g}_{ij}$ is nonnegative, c.f. \cite[Theorem 2.4. and Lemma 3.8]{And75}. In particular, $\sigma_{11}^{f,g}$ is a nonnegative Radon measure. 

For any fixed $F\in \mathcal{V}^{f,g}$ a linear functional $L_F^{f,g}$ on $\mathcal{V}$ is defined by
\begin{equation}\label{E:LFf}
L_F^{f,g}(H):=2E^{f,g}(FH,F)-E^{f,g}(F^2,H), \ \ H\in \mathcal{V}.
\end{equation}
Writing $L_F^{f,g,(c)}(H):=N^{f,g}(FH,F)-N^{f,g}(F^2,H)$
to denote the strongly local part of $L_F^{f,g}$ we obtain
\begin{equation}\label{E:intLFf}
L_F^{f,g}(H)=L_F^{f,g,(c)}(H)+\int H(x)\int (F(x)-F(y))^2 J^{f,g}(dxdy)+\int HF^2\:d\kappa^{f,g}.
\end{equation}
The representation (\ref{E:measurerep}) implies that for any $F\in C_c^1(I^{f,g}_0)$ we have 
\[L_F^{f,g,(c)}(H)=2\sum_{ij}\int H\:\frac{\partial F}{\partial x_i}\frac{\partial F}{\partial x_j}\:d\sigma^{f,g}_{ij}, \ H\in C^1_c(I^{f,g}_0),\]
and as a consequence, $L_F^{f,g}$ is seen to be a positive linear functional on $C_c^1(I_0^{f,g})$. In Lemma \ref{L:mainstatement} below we slightly extend this positivity property to certain functions that not necessarily vanish in a neighborhood of zero. As a preparation, we discuss the supports of the representing measures. Let $\varepsilon>0$  be sufficiently small such that  
\begin{equation}\label{E:Keps}
K:=[-\left\|f\right\|_{\sup}-\varepsilon,\left\|f\right\|_{\sup}+\varepsilon]\times [-\left\|g\right\|_{\sup}-\varepsilon,\left\|g\right\|_{\sup}+\varepsilon]\subset I^{f,g}.
\end{equation}

\begin{lemma}\label{L:supports}
For any $\varepsilon>0$ the supports $\supp\:\sigma^{f,g}_{ij}$ and $\supp\:\kappa^{f,g}$ are contained in $K$ and $\supp J^{f,g}$ is contained in $K\times K$. 
\end{lemma}

\begin{proof}
For all $F,G\in C_c^1(I^{f,g}_0)$ with $\supp\:F\subset I^{f,g}\setminus K$ or  $\supp\:G\subset I^{f,g}\setminus K$ we have $E^{f,g}(F,G)=0$. Varying $F$ shows immediately that $\supp\:\kappa^{f,g}\subset K$. If $U$ and $V$ are disjoint open subsets of $I^{f,g}\setminus K$, and if $F$ and $G$ are supported in $U$ and $V$, respectively, then 
\[\int\int F(x)G(y)\:J^{f,g}(dxdy)=0,\]
what implies that $J^{f,g}$ vanishes on $U\times V$. Hence $J^{f,g}$ must vanish on $I^{f,g}\setminus K \times I^{f,g}\setminus K$. We may proceed similarly to show $J^{f,g}$ vanishes on $I^{f,g}\setminus K \times K$. For the strongly local part the statement follows from the straightforward identities
\begin{equation}\label{E:Anderssontrick}
2\int F\:d\sigma^{f,g}_{ij}=N^{f,g}(x_i F, x_j\theta)+N^{f,g}(x_jF, x_i\theta)-N^{f,g}(F, x_ix_j\theta),
\end{equation}
valid for any $F\in C_c^1(I_0^{f,g})$ and any $\theta\in C_c^1(I^{f,g}_0)$ with $\theta\equiv 1$ on a neighborhood of $\supp\:F$, see \cite[p. 24]{And75}. If $\supp\:F\subset I^{f,g}\setminus K$, then the right hand side of (\ref{E:Anderssontrick}) vanishes.
\end{proof}

Now we consider functions $F(x)=F(x_1)$ and $H(x)=H(x_2)$ depending only on the first and second variable, respectively.

\begin{lemma}\label{L:mainstatement}
Let $F\in C^1(I^{f,g})$ be a function with $F(x_1,x_2)=F(x_1)$ such that $F'\geq 0$ and $F(0)=0$. Let $H\in \lip(I_0^{f,g})$ be a function with $H(x_1,x_2)=H(x_2)$ such that $H\geq 0$ and $H(0)=0$. Then we have 
\begin{equation}\label{E:readyforproj}
L^{f,g,(c)}_F(H)=\int H(x_2)F'(x_1)^2 \sigma_{11}^{f,g}(dx)
\end{equation}
and therefore 
\begin{equation}\label{E:LFfpos}
L_F^{f,g}(H)\geq 0.
\end{equation}
\end{lemma}

\begin{proof}
Given a function $F$ as in the lemma we may always assume it has compact support in $I^{f,g}$ (otherwise multiply with a simple cut-off). For any $n$ let $\varphi_n:\mathbb{R}\to [0,1]$ be the continuous function that vanishes in $(-\frac{1}{n},\frac{1}{n})$, equals one outside $(-\frac{2}{n},\frac{2}{n})$ and is linear in between.
Define functions $F_n\in C_c^ 1(I_0^ f)$ by 
\begin{equation}\label{E:Fn}
F_n(x_1):=\int_{0}^{x_1}\varphi_n(t)F'(t)dt-\int_{x_1}^{0}\varphi_n(t)F'(t)dt.
\end{equation}
Then $\lim_n F_n'=F'$ monotonically and $\lim_n F_n=F$ monotonically. By monotone convergence therefore
\[\lim_n \int F_n'^2 d\sigma_{11}^{f,g}=\int F'^2 d\sigma_{11}^{f,g}\ \ \text{ and }\ \ \lim_n \int F_n^2 d\kappa{f,g}=\int F^2 d\kappa^{f,g}.\]
Similarly, using the symmetry of $J^{f,g}$, 
\begin{multline}
\lim_n \int\int (F_n(x_1)-F_n(y_1))^2J^ {f,g}(dxdy)=2\lim_n\int\int_{\left\lbrace x_1<y_1\right\rbrace}\left(\int_{x_1}^{y_1} F_n'(t)dt\right)^2J^{f,g}(dxdy)\notag\\
=\int\int(F(x_1)-F(y_1))^2J^{f,g}(dxdy).
\end{multline}
Now let $\widetilde{E}^{f,g}(F)$ denote the right hand side of (\ref{E:BD2}) with $F$ as in the lemma. The bound $\sup_n \left\|F_n'\right\|_{\sup}\leq \left\|F'\right\|_{\sup}$  implies 
$\widetilde{E}^{f,g}(F)\leq \sup_nE^{f,g}(F_n)\leq\left\|F'\right\|_{\sup}\mathcal{E}(f)<+\infty$. Therefore 
\begin{multline}
\widetilde{E}^{f,g}(F-F_n)=2\int(1-\varphi_n)^2F'^2d\sigma_{11}^{f,g}\notag\\
+2\int\int_{\left\lbrace x_1<y_1\right\rbrace}\left(\int_{x_1}^{y_1}(1-\varphi_n(t))F'(t) dt\right)^2J^f(dxdy)
+\int(F-F_n)^2d\kappa^f
\end{multline}
converges to zero by dominated convergence. In particular, $(F_n)_n$ is $E^{f,g}$-Cauchy. On the other hand $\lim_n F_n=F$ uniformly by bounded convergence. Consequently the sup-norm-closability of $(\mathcal{E},\mathcal{D})$ implies $\lim_n\mathcal{E}(F-F_n)=0$. Given $H\in C_c^1(I_0^{f,g})$ we have in particular $\lim_n N^{f,g}((F-F_n)^2,H)=0$ by contractivity and Cauchy-Schwarz. If in addition $H(x_1,x_2)=H(x_2)$ then, since $FH\in C_c^1(I_0^{f,g})$, formula 
(\ref{E:measurerep}) yields
\begin{multline}
N^{f,g}(FH-F_nH)
=\int(1-\varphi_n(x_1))^2 F'^2(x_1)H^2(x_2)\:d\sigma^{f,g}_{11}\notag\\
+2\int(1-\varphi_n(x_1))F'(x_1)H(x_2)(F(x_1)-F_n(x_1))H'(x_2)\:d\sigma_{12}^{f,g}\notag\\
+\int(F(x_1)-F_n(x_1))^2 H'^2(x_2)\:d\sigma_{22}^{f,g},
\end{multline}
what converges to zero by bounded convergence. Therefore $\lim_n N^{f,g}((F-F_n)H, F-F_n)=0$ and consequently $\lim_n L^{f,g,(c)}_{F-F_n}(H)=0$. An estimate analogous to (\ref{E:bilinest}) and monotone convergence now show
\[L_F^{f,g,(c)}(H)=\lim_n L_{F_n}^{f,g,(c)}(H)=\lim_n \int H(x_2)F_n'^2(x_1)\sigma_{11}^{f,g}(dx)=\int H(x_2)F'^2(x_1)\sigma_{11}^{f,g}.\]
Together with (\ref{E:intLFf}) we obtain (\ref{E:LFfpos}). For general nonnegative $H\in Lip(I^f)$ with $H(0)=0$ we can obtain (\ref{E:readyforproj}) and (\ref{E:LFfpos}) using monotone convergence and (\ref{E:intLFf}). 
\end{proof}

We prove Theorem \ref{T:positive}.
\begin{proof}
Consider the projections $\pi_i:\mathbb{R}^2\to \mathbb{R}$, $\pi_i(x_1,x_2)=x_i$, $i=1,2$, and apply Lemma \ref{L:mainstatement} with $\pi_1$ in place of $F$ and $\pi_2^+:=\pi_2\vee 0$ in place of $H$. If $f\in\mathcal{D}$ and $g\in\mathcal{D}^+$ then 
\[L_f(g)=2\mathcal{E}(fg,f)-\mathcal{E}(f^2,g)=2E^{f,g}(\pi_1\pi_2^+,\pi_1)-E^{f,g}(\pi_1^2,\pi_2^+)=L^{f,g}_{\pi_1}(\pi_2^+)\geq 0.\]
Varying $\varepsilon>0$ in (\ref{E:Keps}) we can also see that for any $f,g\in \mathcal{D}$, 
\[|L_f(g)|=|L^{f,g}_{\pi_1}(\pi_2)|\leq 2\left\|g\right\|_{\sup}\:E^{f,g}(\pi_1)=2\left\|g\right\|_{\sup}\:\mathcal{E}(f),\]
i.e. (\ref{E:Lagrangesupbound}). For $F\in C^1(\mathbb{R})$ with $F(0)=0$ and $\left\|F'\right\|_{\sup}\leq 1$ the contraction property (\ref{E:LagrangeMarkov}) is a direct consequence of (\ref{E:intLFf}) and (\ref{E:readyforproj}). Given a general normal contraction $F$, we have $|F'|\leq 1$ a.e. on $\mathbb{R}$. If $(F_n)_n$ is a sequence of functions defined as in (\ref{E:Fn}) then similarly as before $\lim_n F_n(f)=F(f)$ uniformly on $X$ and $(F_n(f))_n$ is $\mathcal{E}$-Cauchy. Sup-norm-closability and (\ref{E:Lagrangesupbound}) imply that for any $g\in\mathcal{D}$ we have $\lim_n L_{F_n(f)-F(f)}(g)=0$ and by (\ref{E:bilinest}), $L_{F(f)}(g)=\lim_n L_{F_n(f)}(g)$. This implies (\ref{E:LagrangeMarkov}).
\end{proof}

\section{Bilinear forms with zero killing}\label{S:zerokill}

In this section we discuss bilinear forms with zero killing and record some consequences for the linear functionals (\ref{E:Lagrange1}). To define the killing functional it suffices to consider lattice properties and the unit contraction, to deduce some statements for bilinear forms we will additionally assume to deal with an algebra.

Let $\mathcal{D}$ be a vector lattice (with respect to the natural pointwise order) of bounded real valued functions endowed with the supremum norm and again let $\mathcal{D}^+$ denote its nonnegative elements. We assume that $\mathcal{D}$ has the Stone property and therefore is stable under the unit contraction $T_1$.

Set $\mathcal{D}^+_0:=\left\lbrace f\in \mathcal{D}^+: E_f\neq \emptyset\right\rbrace$,
where for a given function $f\in\mathcal{D}^+$ 
\begin{equation}\label{E:Ef}
E_f:=\left\lbrace \varphi\in\mathcal{D}: \mathbf{1}_{\left\lbrace f>0\right\rbrace}\leq \varphi\leq \mathbf{1}\right\rbrace.
\end{equation}
It is not difficult to see that $\mathcal{D}^+_0$ is an ideal of the lattice cone $\mathcal{D}^+$. Since $E_{T_1(f)}=E_f$ for any $f\in \mathcal{D}^+$, the unit contraction operates also on $\mathcal{D}^+_0$. Set $\mathcal{D}_0:=\lin\left(\mathcal{D}^+_0\right)$. Clearly $\mathcal{D}_0$ is a subspace of $\mathcal{D}$, and if $\mathbf{1}\in\mathcal{D}$ then $\mathcal{D}_0=\mathcal{D}$. Our notation is consistent in the sense that $\mathcal{D}_0^+$ is exactly the cone of nonnegative elements of $\mathcal{D}_0$, as can be seen using the lattice structure that $\mathcal{D}_0$ inherits from $\mathcal{D}$.

We use a standard decomposition for functions from \cite[p. 6]{Allain}. Recall the definition of the contractions $T_\alpha$, $\alpha\in\mathbb{R}$, from Section \ref{S:main}. Given $f\in\mathcal{D}$, let $N$ be the smallest integer greater than $\left\|f\right\|_{\sup}$. For $n=1,2,...$ set 
\begin{align}\label{E:tranches}
f_{k,n}&:=T_{\frac{k+1}{2^n}}(f) - T_{\frac{k}{2^n}}(f),\ \ k=0,1,\dots, 2^n N-1,\text{ and }\\
f_{k,n}&:=T_{\frac{k}{2^n}}(f) - T_{\frac{k+1}{2^n}}(f),\ \ k=-2^n N,\dots, -1. \notag
\end{align}
Then $f=\sum_{k=-2^nN}^{2^nN-1} f_{k,n}$ and $\left\|f_{k,n}\right\|_{\sup}\leq 2^{-n}$. The following is an immediate consequence.

\begin{lemma}\label{L:D0dense}
The vector space $\mathcal{D}_0$ is uniformly dense in $\mathcal{D}$.
\end{lemma}

\begin{proof}
Let $f\in\mathcal{D}$, we may assume $f\in\mathcal{D}^+$. For any $n$ the bound $f-\sum_{k=1}^ {2^nN-1} f_{k,n}=f_{0,n}\leq 2^{-n}$ holds, and for any $k\geq 1$ we have $f_{k,n}\in \mathcal{D}_0$ because $\mathbf{1}_{\left\lbrace f_{k,n}>0\right\rbrace}\leq 2^n f_{k-1,n}\leq \mathbf{1}$.
\end{proof}

\begin{remark}\label{R:unitcontract}
If the unit contraction operates on $(\mathcal{E},\mathcal{D})$, then the bilinearity of $\mathcal{E}$ implies the following facts, cf. \cite[p. 2]{Allain}:
\begin{enumerate}
\item[(i)] For $f\in\mathcal{D}^+$ and $g\in E_f$ we have $\mathcal{E}(f,g)\geq 0$. 
\item[(ii)] For $f,g\in \mathcal{D}^+$ with $f\wedge g=0$ we have $\mathcal{E}(f,g)\leq 0$. 
\end{enumerate}
\end{remark}

Clearly any bilinear form $(\mathcal{E},\mathcal{D})$ induces a bilinear form $(\mathcal{E},\mathcal{D}_0)$ by restriction of $\mathcal{E}$ to $\mathcal{D}_0\times\mathcal{D}_0$, and if the unit contraction operates on $(\mathcal{E},\mathcal{D})$ then also on $(\mathcal{E},\mathcal{D}_0)$.


Let $(\mathcal{E},\mathcal{D})$ be a bilinear form on which the unit contraction operates. For $f\in \mathcal{D}^+_0$ set
\begin{equation}\label{E:killing}
K(f):=\inf\left\lbrace \mathcal{E}(f,\varphi): \varphi \in E_f\right\rbrace.
\end{equation}

For simplicity we refer to an additive and positively homogeneous functional on a cone as a \emph{linear functional}, and we call it \emph{positive} if it takes values in $[0,+\infty)$.

\begin{lemma}\label{L:killing}
Formula (\ref{E:killing}) defines a positive linear functional $K$ on $\mathcal{D}^+_0$.
\end{lemma}

\begin{proof}
Obviously the functional $K$ is positively homogeneous. To verify its additivity, let $f,g\in \mathcal{D}^+_0$. Given $\varepsilon>0$ we can find
$\varphi\in E_{f+g}$ such that 
\[K(f+g)+\varepsilon>\mathcal{E}(f+g,\varphi)=\mathcal{E}(f,\varphi)+\mathcal{E}(g,\varphi).\]
Since $\mathbf{1}_{\left\lbrace f>0\right\rbrace}\vee\mathbf{1}_{\left\lbrace g>0\right\rbrace}=\mathbf{1}_{\left\lbrace f+g>0\right\rbrace}\leq \varphi\leq \mathbf{1}$ we have $\varphi\in E_f\cap E_g$, and therefore
\[K(f+g)+\varepsilon\geq K(f)+K(g).\]
Now we may let $\varepsilon$ tend to zero. If, on the other hand, $\varphi_1\in E_f$ and $\varphi_2\in E_g$, then $\mathbf{1}_{\left\lbrace f+g>0\right\rbrace}\leq \varphi_1\vee\varphi_2\leq \mathbf{1}$, hence $\varphi_1\vee \varphi_2\in E_{f+g}$
and therefore
\[K(f+g)\leq \mathcal{E}(f+g,\varphi_1\vee\varphi_2)=\mathcal{E}(f,\varphi_1\vee\varphi_2)+\mathcal{E}(g,\varphi_1\vee\varphi_2)\leq\mathcal{E}(f,\varphi_1)+\mathcal{E}(g,\varphi_2),\]
note that $\mathcal{E}(f,\varphi_1\vee \varphi_2-\varphi_1)\leq 0$ due to Remark (\ref{R:unitcontract}) (ii), because $(\varphi_1\vee \varphi_2)(x)-\varphi_1(x)=0$ for any $x\in X$ with $f(x)>0$. Similarly for the other summand. Taking infima yields
\[K(f+g)\leq K(f)+K(g).\]
Therefore $K$ is additive. By Remark (\ref{R:unitcontract}) (i) the linear functional $K$ is positive.
\end{proof}

For general $f\in\mathcal{D}_0$ let $K$ be defined by $K(f):=K(f_+)-K(f_-)$. We refer to $K$ as the \emph{killing functional} of $(\mathcal{E},\mathcal{D})$ and say that 
$(\mathcal{E},\mathcal{D}_0)$ has \emph{zero killing} if $K(f)=0$ for all $f\in\mathcal{D}_0$.
Note that if $\mathbf{1}\in\mathcal{D}$ then $K(f)$ is defined for all $f\in\mathcal{D}$ and $(\mathcal{E},\mathcal{D})$ has zero killing if and only if $\mathcal{E}(\mathbf{1})=0$.
 
We mostly work under the additional assumption that $\mathcal{D}$ is an algebra. Recall that this is the case if and only if $f\in\mathcal{D}$ implies $f^2\in\mathcal{D}^+$. It is easy to see that if $\mathcal{D}$ is an algebra, then the space $\mathcal{D}_0$ is an ideal of  $\mathcal{D}$, and in particular, $f^2\in\mathcal{D}^+_0$ for any $f\in\mathcal{D}_0$.

\begin{proposition}\label{P:splitoffkill}
Let $\mathcal{D}$ be an algebra and $(\mathcal{E},\mathcal{D})$ a bilinear form on which the unit contraction operates. Let $K$ be its killing functional. Then 
\begin{equation}\label{E:dominateK}
K(f^2)\leq \mathcal{E}(f),\ \ f\in\mathcal{D}_0,
\end{equation}
and $\mathcal{Q}(f,g):=\mathcal{E}(f,g)-K(fg)$, $f,g\in\mathcal{D}_0$,
defines a bilinear form $(\mathcal{Q},\mathcal{D}_0)$ on which the unit contraction operates and with zero killing.
\end{proposition}

\begin{proof}
To verify (\ref{E:dominateK}) we use again the decomposition (\ref{E:tranches}). Note that $\left\lbrace |f_{k+1,n}|> 0\right\rbrace\subset \left\lbrace 2^n|f|\geq k+1\right\rbrace\subset\left\lbrace 2^n|f_{k,n}|=1\right\rbrace$,
hence $2^n |f_{k,n}|\in E_{|f_{k+1,n}|}$ for all $n$ and $k$. By the definition of $K$ therefore
\begin{align}
\mathcal{E}(f)=\sum_k\sum_l \mathcal{E}(f_{k,n},f_{l,n})
&=\sum_k\left(\sum_{l\leq k-1}\mathcal{E}(f_{k,n},f_{l,n})+\sum_{l\geq k+1}\mathcal{E}(f_{k,n},f_{l,n})\right)+\sum_k\mathcal{E}(f_{k,n})\notag\\
&\geq \sum_k\left(\sum_{l\leq k-1} 2^{-n} K(f_{l,n})+\sum_{l\geq k+1} 2^{-n}K(f_{k,n})\right)\notag\\
&=\sum_k\sum_{k\neq l} K(f_{k,n}f_{l,n}).\notag
\end{align}
The last expression tends to $K(f^2)$ as $n$ goes to infinity. This from the uniform bound
\[f^2-\sum_k\sum_{l\neq k} f_{k,n}f_{l,n}=\sum_{k=-2^nN}^{2^nN-1} f_{k,n}^2\leq 2^{-n+1}N,\]
and together with the positivity of $K$, this yields
\[K(f^2-\sum_k\sum_{l\neq k} f_{k,n}f_{l,n})=K(\varphi(f^2-\sum_k\sum_{l\neq k} f_{k,n}f_{l,n}))\leq 2^{-n+1}N K(\varphi)\]
for arbitrary $\varphi\in E_{f^2}$. Clearly $(\mathcal{Q},\mathcal{D}_0)$ is symmetric and bilinear, and by (\ref{E:dominateK}) it is nonnegative definite. To see that the unit contraction operates, note that by the definition of $K$ we have
\[\mathcal{E}(f-T_1(f), T_1(f))\geq K((f-T_1(f))T_1(f)),\]
recall that $T_1(f)\in E_{f-T_1(f)}$. Together with (\ref{E:dominateK}) this yields
\begin{align}
\mathcal{E}(f)-\mathcal{E}(T_1(f))=\mathcal{E}(f-T_1(f),f+T_1(f))
&=\mathcal{E}(f-T_1(f),f-T_1(f)) + 2\mathcal{E}(f-T_1(f), T_1(f))\notag\\
&\geq K((f-T_1(f))^2)+2K((f-T_1(f))T_1(f))\notag\\
&=K((f-T_1(f))(f+T_1(f)))\notag\\
&=K(f^2)-K(T_1(f)^2),\notag
\end{align}
which is $\mathcal{Q}(T_1(f))\leq \mathcal{Q}(f)$. To see that $\mathcal{Q}$ has zero killing let $f\in \mathcal{D}_0^+$ and let $\varepsilon>0$ be arbitrary. Then by Remark \ref{R:unitcontract} (i) and the definition of $K$ there exists some $\varphi\in E_f$ such that $0\leq\mathcal{Q}(f,\varphi)=\mathcal{E}(f,\varphi)-K(f)<\varepsilon$.
\end{proof}

In the case of zero killing we can use lattice properties to recover $\mathcal{E}$ from the functionals $L_f$, $f\in\mathcal{D}$, as defined in $(\ref{E:Lagrange1})$. The next lemma does not need the $L_f$'s to be positive.

\begin{lemma}
Let $\mathcal{D}$ be an algebra and $(\mathcal{E},\mathcal{D})$ be a bilinear form such that $(\mathcal{E},\mathcal{D}_0)$ has zero killing. Then we have 
\begin{equation}\label{E:suprep}
\mathcal{E}(f)=\frac12\sup\left\lbrace L_f(\varphi): \varphi\in E_{f^2}\right\rbrace
\end{equation}
for any $f\in\mathcal{D}_0$.
\end{lemma}

\begin{proof}
To verify (\ref{E:suprep}) let $f\in\mathcal{D}_0$ and let $\varepsilon>0$ be arbitrary. Then $f^2\in\mathcal{D}_0^+$, and there is some $\varphi\in E_{f^2}$ such that $0\leq \mathcal{E}(f^2,\varphi)\leq \varepsilon$, and since $f\varphi=f$, formula (\ref{E:Lagrange1}) yields $2\mathcal{E}(f)\leq L_f(\varphi)+\varepsilon\leq 2\mathcal{E}(f)+\varepsilon$.
\end{proof}

However, if the functionals $L_f$ are positive, $\mathcal{E}(f)$ is seen to be half the norm of $L_f$. In particular, for a bilinear form with zero killing the upper bound in (\ref{E:Lagrangesupbound}) is sharp.

\begin{corollary}\label{C:functpos}
Let $(\mathcal{E},\mathcal{D})$ be a bilinear form such that $(\mathcal{E},\mathcal{D}_0)$ has zero killing and all $L_f$, $f\in\mathcal{D}_0$, are positive linear functionals. Then for any $f\in\mathcal{D}_0$,
\begin{equation}\label{E:normrep}
\mathcal{E}(f)=\frac12\left\|L_f\right\|_{\mathcal{D}'}.
\end{equation}
\end{corollary}

\begin{proof}
By the positivity of $L_f$ formula (\ref{E:suprep}) yields $2\mathcal{E}(f)=\sup\left\lbrace L_f(\varphi): \varphi\in\mathcal{D},\ 0\leq \varphi\leq \mathbf{1}\right\rbrace$,
because if $\varphi(x)<\mathbf{1}_{\left\lbrace f\neq 0\right\rbrace}(x)$ for some $x\in X$ then we may just replace $\varphi$ by $\varphi\vee\psi\geq \varphi$ for arbitrary $\psi\in E_{f^2}$, resulting in $L_f(\varphi\vee\psi)\geq L_f(\varphi)$. Again by positivity, this equals
\[\sup\left\lbrace L_f(\varphi): \varphi\in \mathcal{D}, \  \left\|\varphi\right\|_{\sup}\leq 1\right\rbrace.\] 
\end{proof}

\section{Lagrangians}\label{S:Lagrange}

We investigate some properties of Lagrangians. Let $(L,\mathcal{D})$ be a Lagrangian, cf. Definition \ref{D:Lagrange}. Similarly as in Remark \ref{R:unitcontract} we can deduce some immediate consequences of the contraction property $L_{T_1(f)}\leq L_f$, $f\in\mathcal{D}$.

\begin{remark}\label{R:unitcontractL}
Given $h\in\mathcal{D}^+$ the following properties hold:
\begin{enumerate}
\item[(i)] For $f\in\mathcal{D}^+$ and $g\in E_f$ we have $L_{f,g}(h)\geq 0$.
\item[(ii)] For $f,g\in\mathcal{D}^+$ with $f\wedge g=0$ we have $L_{f,g}(h)\leq 0$.
\end{enumerate}
\end{remark}

Let $\mathcal{D}_0\subset \mathcal{D}$ be as in Section \ref{S:zerokill}. A Lagrangian $(L,\mathcal{D})$ induces a Lagrangian $(L,\mathcal{D}_0)$ by restriction, note that $\mathcal{D}'\subset \mathcal{D}_0'$. The contraction properties are inherited. For any $h\in\mathcal{D}^+$ and any $f\in\mathcal{D}^+_0$ set
\[K^h(f)=\inf\left\lbrace L_{f,\varphi}(h):\varphi\in E_f\right\rbrace.\]
Thanks to Remark \ref{R:unitcontractL} we can follow similar arguments as in Lemma \ref{L:killing} to see that $K^h$ is a positive linear functional, and we can define $K^h$ on all of $\mathcal{D}_0$ by linearity. We say that $(L,\mathcal{D}_0)$ has \emph{zero killing} if for any $h\in\mathcal{D}^+_0$ and any  $f\in\mathcal{D}_0$ we have $K^h(f)=0$. Adapting the proof of Proposition \ref{P:splitoffkill} we see that 
$L_f(h)\leq K^h(f)$ for any $h\in\mathcal{D}^+$ and any $f\in\mathcal{D}_0$ and that moreover, 
\begin{equation}\label{E:Ltilde}
\widetilde{L}_{f,g}(h):=L_{f,g}(h)-K^h(fg)
\end{equation}
defines a Lagrangian $(\widetilde{L},\mathcal{D}_0)$ with zero killing. Note that the Lagrangian generated by $(\mathcal{E},\mathcal{D}_0)$ agrees (in $\mathcal{D}'$) with the restriction $(L,\mathcal{D}_0)$ of the Lagrangian $(L,\mathcal{D})$ generated by $(\mathcal{E},\mathcal{D})$.

\begin{remark}
Obviously we may rephrase Theorem \ref{T:positive} by saying that if normal contractions operate on a sup-norm closable bilinear form on a $C_c^1(\mathbb{R}^2)$-stable algebra then it generates a Lagrangian on which normal contractions operate.
\end{remark}

\begin{corollary}\label{C:zerokill} Let $(L, \mathcal{D})$ be a Lagrangian generated by a bilinear form $(\mathcal{E},\mathcal{D})$.
\begin{enumerate}
\item[(i)] If $(\mathcal{E},\mathcal{D}_0)$ has zero killing, then $(L,\mathcal{D}_0)$ has zero killing.
\item[(ii)] If $(Q,\mathcal{D}_0)$ is the bilinear form with zero killing obtained from $(\mathcal{E},\mathcal{D})$ as in Proposition \ref{P:splitoffkill}, then $(\widetilde{L},\mathcal{D}_0)$ defined in (\ref{E:Ltilde}) is the Lagrangian generated by $(Q,\mathcal{D}_0)$. For the respective killing functionals we have $K^h(f^2)=K(f^2h)$, $f\in\mathcal{D}_0$, $h\in \mathcal{D}^+_0$.
\end{enumerate}
\end{corollary}

\begin{proof}
To see (i), let $f\in\mathcal{D}_0$, $h\in\mathcal{D}_0^+$ and $\psi\in E_h$. By Remark \ref{R:unitcontractL} (i) we have
\[0\leq L_{f^2,\varphi}(h)\leq \left\|h\right\|_{\sup} L_{f^2,\varphi}(\psi)\leq \left\|h\right\|_{\sup} L_{f^2,\varphi}(\psi\vee \varphi)\]
for any $\varphi\in E_{f^2}$, and using Remark \ref{R:unitcontract} (i) and the symmetry of $\mathcal{E}$, 
\[L_{f^2,\varphi}(\varphi\vee \psi)=\mathcal{E}(f^2(\psi\vee\varphi),\varphi)+\mathcal{E}(\varphi(\psi\vee\varphi),f^2)-\mathcal{E}(f^2\varphi,\psi\vee\varphi)\leq\mathcal{E}(f^2,\varphi).\]
Taking infima yields (i). For (ii) let $(L^\mathcal{Q},\mathcal{D}_0)$ denote the Lagrangian generated by $(\mathcal{Q},\mathcal{D}_0)$ and let $K$ denote the killing functional as in Proposition \ref{P:splitoffkill}. Given $f\in\mathcal{D}_0$, $\varphi\in E_{f^2}$ and $h\in\mathcal{D}_0^+$, we have
\[\widetilde{L}_{f^2,\varphi}(h)+K^h(f^2)=L_{f^2,\varphi}(h)=L^{\mathcal{Q}}_{f^2,\varphi}(h)+K(f^2h).\]
Taking infima shows $K^h(f^2)=K(f^2h)$ and therefore $(\widetilde{L},\mathcal{D}_0)=(L^\mathcal{Q},\mathcal{D}_0)$.
\end{proof}

Similarly as going from a bilinear form to a Lagrangian we can go from a Lagrangian to a bilinear form. Recall from Section \ref{S:main} that to $(\mathcal{E}_L,\mathcal{D})$ with
\begin{equation}\label{E:EL}
\mathcal{E}_L(f):=\frac{1}{2}\left\|L_f\right\|_{\mathcal{D}'}\ ,\ \ f\in\mathcal{D},
\end{equation}
we refer as the \emph{energy form associated with $(L,\mathcal{D})$}. Note that if $\mathbf{1}\in\mathcal{D}$, then the positivity of $L_f$ implies $\mathcal{E}(f)=L_f(\mathbf{1})$. The functional $f\mapsto \mathcal{E}_L$ satisfies the parallelogram identity and inherits the contraction properties.

\begin{lemma}
We have
\[2((\mathcal{E}_L(f)+\mathcal{E}_L(g))=\mathcal{E}_L(f+g)+\mathcal{E}_L(f-g)\]
for any $f,g\in\mathcal{D}$. The unit contraction operates on $(\mathcal{E}_L,\mathcal{D})$. If normal contractions operate on $(L,\mathcal{D})$, then also on $(\mathcal{E}_L,\mathcal{D})$.
\end{lemma}

Accordingly, polarization produces an associated bilinear form $(\mathcal{E}_L,\mathcal{D})$,
\[\mathcal{E}_L(f,g)=\frac{1}{4}(\mathcal{E}_L(f+g)-\mathcal{E}_L(f-g)).\]

\begin{proof}
For any $\varepsilon>0$ we can find $\varphi_1\in E_{f^2}$ and $\varphi_2\in E_{g^2}$ such that 
\[2(\mathcal{E}_L(f)+\mathcal{E}_L(g))\leq L_f(\varphi_1)+L_g(\varphi_2)+\varepsilon.\]
By positivity and bilinearity this is less or equal
\begin{align}
L_f(\varphi_1\vee \varphi_2)+L_g(\varphi_1\vee \varphi_2)+\varepsilon &=\frac12\left(L_{f+g}(\varphi_1\vee\varphi_2)+L_{f-g}(\varphi_1\vee\varphi_2)\right)+\varepsilon\notag\\
&\leq\frac12\left(\mathcal{E}_L(f+g)+\mathcal{E}_L(f-g)\right)+\varepsilon\notag
\end{align}
since $\varphi_1\vee \varphi_2\in E_{(f+g)^2}\cap E_{(f-g)^2}$. Conversely, given $\varepsilon>0$ there are $\varphi_1\in E_{(f+g)^2}$ and $\varphi_2\in E_{(f-g)^2}$ with
\begin{align}
\mathcal{E}_L(f+g)+\mathcal{E}_L(f-g)&\leq L_{f+g}(\varphi_1)+L_{f-g}(\varphi_2)+\varepsilon\notag\\
&\leq L_{f+g}(\varphi_1\vee\varphi_2)+L_{f-g}(\varphi_1\vee \varphi_2)+\varepsilon\notag\\
&=2\left(L_f(\varphi_1\vee\varphi_2)+L_g(\varphi_1\vee\varphi_2)\right)+\varepsilon\\
&\leq 2\left(\mathcal{E}_L(f)+\mathcal{E}_L(g)\right)+\varepsilon,\notag
\end{align}
note that $\varphi_1\vee\varphi_2\in E_{f^2}\cap E_{g^2}$. In fact, if $x\in X$ is such that $f(x)^2>0$ then $(f+g)^2(x)>0$ or $g(x)=-f(x)$. In the first case we obtain $\varphi_1(x)=1$ and in the second, $\varphi_2(x)=1$, because $(f-g)^2(x)=4f(x)^2>0$. Similarly for $g$. The  contraction properties follow from (\ref{E:EL}).
\end{proof}

For Lagrangians generated by a bilinear form Corollary \ref{C:functpos} (ii) has the following consequence.

\begin{corollary}\label{C:consistent}
For a Lagrangian $(L,\mathcal{D})$ generated by a bilinear form $(\mathcal{E},\mathcal{D})$ we have $\mathcal{E}_L=\mathcal{E}$ on $\mathcal{D}_0$.
\end{corollary}
\begin{proof}
Let $(\mathcal{Q},\mathcal{D}_0)$ be the form with zero killing as in Corollary \ref{C:zerokill} and $(L^{\mathcal{Q}},\mathcal{D}_0)$ its Lagrangian. Then $\mathcal{Q}_{L^\mathcal{Q}}=\mathcal{Q}$ on $\mathcal{D}_0$ by Corollary \ref{C:functpos}. On the other hand, Corollary \ref{C:zerokill} implies
\[\mathcal{E}_L(f)=\frac12\sup\left\lbrace L_f(h): h\in E_{f^2}\right\rbrace=\frac12\sup\left\lbrace L^{\mathcal{Q}}_f(h): h\in E_{f^2}\right\rbrace - K(f^2) =\mathcal{Q}_{L^\mathcal{Q}}(f)-K(f^2).\]
Together, this proves $\mathcal{E}_L(f)=\mathcal{Q}(f)-K(f^2)=\mathcal{E}(f)$.
\end{proof}

The passage (\ref{E:Lagrange1}) from a bilinear form to a Lagrangian encodes a Leibniz rule.

\begin{definition}
We say that a Lagrangian $(L,\mathcal{D})$ \emph{has the global Leibniz property} if for all $f,g,h\in\mathcal{D}$ we have
\[2\mathcal{E}_L(fg,h)=L_{f,h}(g)+L_{g,h}(f).\]
\end{definition}

The following observation is immediate from (\ref{E:Lagrange1}) and Corollary \ref{C:consistent}.

\begin{corollary}
Let $(L,\mathcal{D})$ be a Lagrangian generated by a bilinear form $(\mathcal{E},\mathcal{D})$. Then $(L,\mathcal{D}_0)$ has the global Leibniz property.
\end{corollary}

The global Leibniz property allows to inherit zero killing to the associated energy. 

\begin{corollary}
Let $(L,\mathcal{D})$ be a Lagrangian such that $(L,\mathcal{D}_0)$ has zero killing and the global Leibniz property. Then $(\mathcal{E}_L,\mathcal{D}_0)$ has zero killing.
\end{corollary}

\begin{proof}
This follows since $\mathcal{E}_L(f^2,\varphi)=2 L_{|f|,\varphi}(|f|)$ for all $f\in\mathcal{D}_0$ and $\varphi\in E_{f^2}$.
\end{proof}

We have already mentioned in Section \ref{S:main} that if $\mathbf{1}\in\mathcal{D}$ then obviously $(L,\mathcal{D})$ is sup-norm-closable if and only if its associated energy is. In the general case the following implication is immediate.

\begin{proposition}\label{P:LtoEL}
Let $(L,\mathcal{D})$ be a sup-norm-closable Lagrangian. Then its associated energy $(\mathcal{E}_L,\mathcal{D})$ is a sup-norm-closable bilinear form. 
\end{proposition}

\begin{proof}
Let $(f_n)_n\subset\mathcal{D}$ be a sequence that is $\mathcal{E}_L$-Cauchy and such that $\lim_n\left\|f_n\right\|_{\sup}=0$. Note first that by (\ref{E:Lagrangesupbound}) $(f_n)_n$ is $L(h)$-Cauchy for any $h\in\mathcal{D}^+$, hence $\lim_n L_{f_n}(h)=0$. Now let $\varepsilon>0$. For any $n$ there exists some $h^{(n)}\in\mathcal{D}$ with $\left\|h^{(n)}\right\|_{\sup}\leq 1$ such that $\mathcal{E}_L(f_n)\leq L_{f_n}(h^{(n)})+\varepsilon/3$. By (\ref{E:bilinest}) and (\ref{E:Lagrangesupbound}) there is some $n_0=n_0(\varepsilon)$ such that $|L_{f_n}(h)-L_{f_m}(h)|<\varepsilon/3$ for all $h\in\mathcal{D}$ with $\left\|h\right\|_{\sup}\leq 1$, provided $m,n\geq n_0$. For any fixed $n\geq n_0$ we have $L_{f_k}(h^{(n)})<\varepsilon/3$ for any large enough $k$, and as we may assume $k\geq n_0$, also $L_{f_n}(h^{(n)})<2\varepsilon/3$. Hence $\mathcal{E}_L(f_n)<\varepsilon$ for any $n\geq n_0$.
\end{proof}

The converse, which completes the proof of Corollary \ref{C:ELtoL}, is studied in Section \ref{S:Gelfand}.

\section{A special case: Dirichlet forms}\label{S:Dirichlet}

We prove Theorem \ref{T:Dirichletform}.

\begin{proof} 
The positivity of $L$, the contraction property and the bound (\ref{E:Lagrangesupbound}) all follow from \cite[Proposition I.4.1.1]{BH91}. To prove sup-norm-closability, let $h\in \mathcal{B}^+$ and suppose $(f_n)_n\subset \mathcal{B}$ is $L(h)$-Cauchy and that $\lim_n f_n=0$ uniformly on $X$. We may suppose $\sup_n L_{f_n}(h)>0$, otherwise the desired result is immediate. Below we will prove that 
\begin{equation}\label{E:claim}
\text{$L_{f_n,\cdot}(h)$ converges to zero weakly on $\mathcal{D}$.}
\end{equation}
Then, given $\varepsilon>0$, choose $n_0=n_0(\varepsilon)$ such that for any $n\geq n_0$ we have 
\[L_{f_n-f_{n_0}}(h)^{1/2}<\frac{\varepsilon}{2\sup_n L_{f_n}(h)}.\]
By (\ref{E:claim}) we have $|L_{f_n,f_{n_0}}(h)|\leq \varepsilon/2$ for any large enough $n$. Bilinearity, the triangle inequality and Cauchy-Schwarz therefore yield
\[|L_{f_n}(h)|\leq |L_{f_n,f_{n_0}}(h)|+L_{f_n}(h)^{1/2}L_{f_n-f_{n_0}}(h)^{1/2}<\varepsilon.\]
To see (\ref{E:claim}), note first that for arbitrary $v\in\mathcal{B}$ we have
\begin{align}
|\mathcal{E}((f_m-f_n)h,v)|&\leq |L_{f_m-f_n,v}(h)|+|L_{h,v}(f_m-f_n)|\notag\\
&\leq \sqrt{2}L_{f_m-f_n}(h)^{1/2}\left\|h\right\|_{\sup}^{1/2}\mathcal{E}(v)^{1/2}+2\left\|f_m-f_n\right\|_{\sup}\mathcal{E}(h)^{1/2}\mathcal{E}(v)^{1/2},\notag
\end{align}
where we have used the global Leibniz property, estimate (\ref{E:Lagrangesupbound}) and Cauchy-Schwarz. Consequently the sequence $(f_nh)_n$ is weakly $\mathcal{E}$-Cauchy. On the other hand also
\[\lim_n\int_Xf_n hv\:d\mu=0,\]
what implies that the $\mu$-classes of $(f_nh)_n$ form a weak Cauchy sequence in the Hilbert space $(\mathcal{F},\mathcal{E}_1)$. Therefore $\sup_n\mathcal{E}_1(f_nh)<+\infty$ by uniform boundedness, and there exist a subsequence $(f_{n_k}h)_n$ and some $w\in \mathcal{F}$ to which the $\mu$-classes of $(f_{n_k}h)_n$ converge weakly in $\mathcal{F}$. Necessarily $w=0$, so that for any $v\in\mathcal{B}$ the quantity
\[|L_{f_n,v}(h)|\leq |\mathcal{E}(f_nh,v)|+|L_{h,v}(f_n)|\leq |\mathcal{E}(f_nh,v)|+2\left\|f_n\right\|_{\sup} \mathcal{E}(h)^{1/2}\mathcal{E}(v)^{1/2}\]
can be made arbitrarily small by choosing $n$ sufficiently large.
\end{proof}

Similarly as before set
\[\mathcal{B}_0:=\lin\left(\left\lbrace f\in\mathcal{B}^+: E_f\neq \emptyset\right\rbrace\right).\]
From Lemma \ref{L:D0dense} we know that $\mathcal{B}_0$ is uniformly dense in $\mathcal{B}$. In the Dirichlet form case it is also energy dense.

\begin{lemma}
The space $\mathcal{B}_0$ is $\mathcal{E}$-dense in $\mathcal{B}$.
\end{lemma}
\begin{proof}
Given $f\in\mathcal{B}$ and $n\geq 1$ we have $f=\sum_{k=-2^nN}^{2^nN-1} f_{k,n}$, where $N$ is the smallest integer greater than $\left\|f\right\|_{\sup}$. All $f_{k,n}$ but $f_{0,n}$ and $f_{-1,n}$ are in $\mathcal{B}_0$. However, by \cite[Theorem 1.4.2 (iv)]{FOT94} the functions
$\sum_{k\neq -1,0} f_{k,n}=f-f_{-1,n}-f_{0,n}$ converge to $f$ in $(\mathcal{F},\mathcal{E}_1)$.
\end{proof}

Together with Corollary \ref{C:consistent} we can therefore observe that any Dirichlet form induces a sup-closable bilinear form.

\begin{corollary}\label{C:Dirichletform}
We have $\mathcal{E}_L=\mathcal{E}$. In particular, $(\mathcal{E},\mathcal{B})$ is a sup-norm closable bilinear form on which normal contractions operate.
\end{corollary}

\section{Energy measures and energy dominance}\label{S:energymeas}

We discuss how to represent the functionals $L_f$ by finitely additive measures. Let $\mathcal{D}$ be an algebra of bounded real valued functions on $X$, having the Stone property and being endowed with the supremum norm. Let $(L,\mathcal{D})$ be a Lagrangian.

We review Hewitt's construction in \cite{Hew51}. Let $\mathcal{S}$ denote the family of all subsets of $X$ of form $\left\lbrace f >0\right\rbrace$, $f\in\mathcal{D}$. We assume that $\mathcal{D}$ contains a strictly positive function $\chi>0$, what implies $X\in\mathcal{S}$.
Set 
\[\mathcal{H}(\mathcal{S}):=\left\lbrace A\subset X: A\subset B \text{ for some } B\in\mathcal{S}\right\rbrace.\]
Note that if $A$ is an element of $\mathcal{H}(\mathcal{S})$ then also any subset of $A$ is an element of $\mathcal{H}(\mathcal{S})$. Given a function $f\in\mathcal{D}$, set
\begin{equation}\label{E:gamma}
\Gamma(f)(B):=\sup\left\lbrace L_f(h): h\in\mathcal{D},\ 0\leq h\leq \mathbf{1}_B\right\rbrace, \ B\in \mathcal{S},
\end{equation}
and for a $A\in\mathcal{H}(\mathcal{S})$ put
\begin{equation}\label{E:gammaast}
\Gamma(f)(A):=\inf\left\lbrace \Gamma(f)(B): B\in\mathcal{S},\ A\subset B\right\rbrace,
\end{equation}
cf. \cite[2.9 Definition]{Hew51}. Formula (\ref{E:gammaast}) defines a $[0,+\infty]$-valued set function $\Gamma(f)$ on $\mathcal{H}(\mathcal{S})$ which is monotone, $\Gamma(f)(A)\leq \Gamma(f)(B)$ if $A\subset B$, and subadditive, $\Gamma(f)(A\cup B)\leq \Gamma(f)(A)+\Gamma(f)(B)$, cf. \cite[2.20 Theorem]{Hew51}.

A set $A\in \mathcal{H}(\mathcal{S})$ is called \emph{$\Gamma(f)$-measurable} if for all $E\in\mathcal{H}(\mathcal{S})$ we have $\Gamma(f)(E)\geq \Gamma(f)(E\cap A)+\Gamma(f)(E\cap A^c)$. Let $\alpha(\mathcal{S})$ denote the algebra of subsets of $X$ generated by $\mathcal{S}$. A standard proof shows that for any $f\in\mathcal{D}$ the family $\mathcal{M}(\Gamma(f))$ is an algebra of subsets of $X$ containing $\alpha(\mathcal{S})$, and $\Gamma(f)$ is finitely additive on $\mathcal{M}(\Gamma(f))$, \cite[2.23 Theorem]{Hew51}.

A corresponding notion of integral has been introduced in 
\cite[2.26 Definition]{Hew51}. Given a nonnegative and bounded $\Gamma(f)$-measurable function $g$ on $X$, set
\[\int_X g(x)\:\Gamma(f)(dx):=\lim_{\left\|\Delta\right\|\to 0}\sum_{i=1}^n \alpha_{i-1}\Gamma(f)(\left\lbrace \alpha_{i-1}<g\leq \alpha_i\right\rbrace),\]
the limit taken along a sequence of partitions $\Delta=\left\lbrace \inf g=\alpha_0<\alpha_1<\alpha_2<\dots <\alpha_n=\sup g\right\rbrace$
with decreasing mesh $\left\|\Delta\right\|:=\max_{i=1,...,n}(\alpha_i-\alpha_{i-1})$. Now if $g\in\mathcal{D}$ is such that 
\begin{equation}\label{E:intneqzero}
\Gamma(f)(\left\lbrace g\neq 0\right\rbrace)<+\infty,
\end{equation} 
then 
\[\int_X g(x)\:\Gamma(f)(dx)=L_f(g),\]
see \cite[2.27 Theorem]{Hew51}. Without condition (\ref{E:intneqzero}) we cannot expect equality in the last formula, \cite[2.29 Remark]{Hew51}. For any $f\in\mathcal{D}$ all functions from $\mathcal{D}$ are $\Gamma(f)$-measurable. This yields the following theorem.

\begin{theorem}\label{T:intrepadd}
Let $(L,\mathcal{D})$ be a Lagrangian and assume that $\mathcal{D}$ contains a strictly positive function $\chi>0$. Then for any $f\in\mathcal{D}$ there exists a finite finitely additive measure $\Gamma(f)$ on $\alpha(\mathcal{S})$ such that 
\begin{equation}\label{E:intrepadd}
L_f(h)=\int_X h\:d\Gamma(f), \ \ h\in\mathcal{D}.
\end{equation}
\end{theorem}

The existence of such $\chi>0$ ensures that all $\Gamma(f)$, $f\in\mathcal{D}$, are finite so that (\ref{E:intneqzero}) is trivially satisfied and also the polarization below makes sense.

\begin{proof}
It suffices to verify that $\Gamma(f)$ is finite. Since $X=\left\lbrace \chi>0\right\rbrace$ we have
\[\Gamma(f)(X)=\sup\left\lbrace L_f(h): h\in\mathcal{D},\ 0\leq h\leq \mathbf{1}\right\rbrace= \left\|L_f\right\|_{\mathcal{D}'}<+\infty.\]
\end{proof}

For any $f,g\in\mathcal{D}$ we can define a finite signed finitely additive measure  $\Gamma(f,g)$ by polarization,
\[\Gamma(f,g)(A):=\frac14\left(\Gamma(f+g)(A)-
\Gamma(f-g)(A)\right),\ \ A\in\alpha(\mathcal{S}).\]
Similarly as for $L_{f,g}$ bilinearity and Cauchy-Schwarz imply 
\begin{equation}\label{E:contmeasure} 
|\Gamma(f)(A)^{1/2}-\Gamma(g)(A)^{1/2}|\leq \Gamma(f-g)(A)^{1/2},\ \ A\in\alpha(\mathcal{S}),
\end{equation}
for any $f,g\in\mathcal{D}$. From (\ref{E:gamma}) and (\ref{E:gammaast}) together with Corollary \ref{C:functpos} we obtain the estimate $\Gamma(f)(A)\leq 2\mathcal{E}_L(f)$,  
 $A\in\alpha(\mathcal{S})$. Now recall Definition \ref{D:energysep}.

\begin{remark}\label{R:energysep}\mbox{}
\begin{enumerate} 
\item[(i)] If there exists a countable set of functions $\left\lbrace f_n\right\rbrace_n\subset \mathcal{D}$ with span $\mathcal{E}_L$-dense in $\mathcal{D}$, then by (\ref{E:Lagrangesupbound}) and (\ref{E:contmeasure}) the Lagrangian $(L,\mathcal{D})$ is energy separable.
\item[(ii)] If $(X,\mathcal{X},\mu)$ is a $\sigma$-finite measure space and $\mathcal{X}$ is generated by a countable semiring, then $L_2(X,\mu)$ is separable, and if $(\mathcal{E},\mathcal{F})$ is a Dirichlet form on $L_2(X,\mu)$, then there is a countable set  of functions $\left\lbrace f_n\right\rbrace_n\subset \mathcal{B}$ that is $\mathcal{E}$-dense in $\mathcal{B}$, hence the bilinear form $(\mathcal{E},\mathcal{B})$ is energy separable. For instance, see the proof of \cite[Theorem 1.4.2 (iii)]{FOT94}.
\end{enumerate}
\end{remark}

In the energy separable case the existence of a finite energy dominant measure and the countable additivity of the set functions $\Gamma(f)$, $f\in\mathcal{D}$, are equivalent.

\begin{theorem}\label{T:energydom}
Let $(L,\mathcal{D})$ be a Lagrangian and assume $\mathcal{D}$ contains a strictly positive function $\chi>0$.
\begin{enumerate}
\item[(i)] If there exists a finite energy dominant measure $m$ for $(L,\mathcal{D})$ then all $\Gamma(f)$, $f\in\mathcal{D}$, are measures on $\alpha(\mathcal{S})$.
\item[(ii)] If $(L,\mathcal{D})$ is energy separable and all $\Gamma(f_n)$ in (\ref{E:standardsum}) are measures  on $\alpha(\mathcal{S})$, then (\ref{E:standardsum}) is a finite energy dominant measure for $(L,\mathcal{D})$.
\end{enumerate}
\end{theorem}

\begin{proof}
(i) holds because if $m$ is a finite energy dominant measure then any $\Gamma(f)$ is a measure, see for instance \cite[Theorem 6.1.11]{BRBR}. (ii) is trivial.  
\end{proof}

\begin{remark}\label{R:lcccase}\mbox{}
\begin{enumerate}
\item[(i)] If $X$ is a locally compact Hausdorff space and $\mathcal{D}$ is a dense subspace of the space $C_c(X)$ of continuous compactly supported functions on $X$, then $\sigma(\mathcal{D})$ is the Borel-$\sigma$-algebra on $X$ and all $\Gamma(f)$, $f\in\mathcal{D}$, are finite Radon measures on $X$. If in addition $(L,\mathcal{D})$ is  energy separable, then (\ref{E:standardsum}) defines a finite Radon measure $m$ on $X$ that is energy dominant for $(L,\mathcal{D})$.
\item[(ii)] If $(\mathcal{E},\mathcal{F})$ is a regular Dirichlet form with core $\mathcal{C}$, then the energy measures $\Gamma(f)$, $f\in\mathcal{C}$, of the bilinear form $(\mathcal{E},\mathcal{C})$ coincide with the energy measures as introduced by Silverstein \cite{S74}, LeJan \cite{LJ78} and Fukushima \cite{FOT94}.
\end{enumerate}
\end{remark}

Carath\'eodory's extension theorem and the Radon-Nikodym theorem imply the following.

\begin{corollary}\label{C:dominant}
Let $(L,\mathcal{D})$ be a Lagrangian and assume $\mathcal{D}$ contains a strictly positive function $\chi>0$. If there exists a finite energy dominant measure $m$ for $(L,\mathcal{D})$ then $m$ and all $\Gamma(f)$, $f\in \mathcal{D}$, extend uniquely to finite measures $m$ and $\Gamma(f)$ on the $\sigma$-algebra $\sigma(\mathcal{S})=\sigma(\mathcal{D})$ generated by $\mathcal{D}$. The measures $\Gamma(f)$, $f\in\mathcal{D}$ are absolutely continuous with respect to $m$ on $\sigma(\mathcal{D})$ and admit $m$-integrable densities
\[\Gamma(f):=\frac{d\Gamma(f)}{dm},\ \ f\in\mathcal{D}.\]
\end{corollary}

\section{Sup-norm closability implies $L_2$-closability}\label{S:clos}

This section contains the proof of Theorem \ref{T:clos}, another version of this theorem and some auxiliary facts.

We first observe that $\mathcal{D}$ is dense in $L_2(X,m)$. 

\begin{lemma}
Let $m$ be a finite measure on $\sigma(\mathcal{D})$. Then $\mathcal{D}$ is dense in $L_2(X,m)$.
\end{lemma}

For any $k\geq 1$ set $A_k:=\left\lbrace \chi>\frac1k\right\rbrace$. Then $A_k\subset A_{k+1}$ for all $k$ and $\chi=\bigcup_{k=1}^\infty A_k$. Setting 
\begin{equation}\label{E:chik}
\chi_k:=k(k+1)\left(\chi\wedge\frac1k-\chi\wedge\frac{1}{k+1}\right), \ \ k\geq 1,
\end{equation}
we obtain functions $\chi_k\in\mathcal{D}^+_0$ such that $0\leq \chi_k\leq 1$, $\chi_k=1$ on $A_k$ and $\chi_k=0$ outside $A_{k+1}$.

\begin{proof} It suffices to prove that any bounded and $\sigma(\mathcal{D})$-measurable function can be approximated in $L_2(X,m)$ by a sequence of functions from $\mathcal{D}$. Consider the cone
\begin{multline}
\mathcal{K}:=\left\lbrace f:X\to [0,\infty) : \text{ $f$ is bounded and the (pointwise) limit}\right.\notag\\
\left.\text{of an increasing sequence $(f_n)_n\subset\mathcal{D}^+$}\right\rbrace.
\end{multline} 
The vector space $\mathcal{H}:=\lin (\mathcal{K})$ generated by $\mathcal{K}$ is a monotone vector system containing $\mathcal{D}$, cf. \cite[Appendix A0.]{Sh88}. In fact, we have $0\leq \chi_1\leq \chi_2\leq \dots$ and $\lim_k\chi_k=\mathbf{1}$ pointwise for the functions $\chi_k$ defined in (\ref{E:chik}). This implies $\mathbf{1}\in\mathcal{K}$. If $f$ is a bounded function on $X$ and there exists a sequence $(f_n)\subset \mathcal{K}$ that monotonically increases to $f$, then $f$ is a member of $\mathcal{K}$: For fixed $n$ let $(\varphi^{(n)}_k)_k$ be a sequence from $\mathcal{D}^+$ that monotonically increases to $f_n$. Let $k(n)$ be such that $f_n-\varphi^{(n)}_{k(n)}<2^{-n}$ and put $\psi_n:=\varphi^{(n)}_{k(n)}$. Then $\lim_n \psi_n(x)=f(x)$ and $\psi_n(x)\leq f(x)$ for all $x\in X$. Setting $g_n:=\max_{k\leq n}\psi_k$ we obtain a sequence $(g_n)_n\subset\mathcal{D}^+$ that monotonically increases to $f$, hence $f\in\mathcal{K}$. By the monotone class theorem the vector space $\mathcal{H}$ therefore contains all bounded $\sigma(\mathcal{D})$-measurable functions, see for instance \cite[Theorem A06]{Sh88}, what implies the desired approximation property.
\end{proof}

\begin{remark}
Every monotone vector system is closed under uniform convergence, \cite[Appendix A]{Sh88}. By Lemma \ref{L:D0dense} we could therefore replace $\mathcal{D}^+$ by $\mathcal{D}_0^+$ in the definition of $\mathcal{K}$ and conclude that $\sigma(\mathcal{D})=\sigma(\mathcal{D}_0)$ and that $\mathcal{D}_0$ is dense in $L_2(X,m)$. As far as closability is concerned, we could then consider $(\mathcal{E},\mathcal{D}_0)$ in place of $(\mathcal{E},\mathcal{D})$ in Theorem \ref{T:clos}.
\end{remark}

In the sequel we will show that for any finite energy dominant measure $m$ the form $(\mathcal{E},\mathcal{D})$ is closable in $L_2(X,m)$. We consider related bilinear forms on the real line. Given $f\in\mathcal{D}$ let $I^f\subset\mathbb{R}$ be a bounded open interval such that $[-\left\|f\right\|_{\sup},\left\|f\right\|_{\sup}]\subset I^f$ and write $I^f_0:=I^f\setminus\left\lbrace 0\right\rbrace$. By $\lip_0(I^f)$ we denote the space of Lipschitz functions $F$ on $I^f$ such that $F(0)=0$. For fixed $f\in\mathcal{D}$ set
\begin{equation}\label{E:coordform}
E^f(F,G):=\mathcal{E}(F(f),G(f))
\end{equation}
for any $F,G\in \lip_0(I^f)$. Obviously $E^f$ is a nonnegative definite symmetric bilinear form on $\lip_0(I^f)$, and normal contractions operate. We use the notation $\diag:=\left\lbrace (x,x):x\in\mathbb{R}_0\right\rbrace$ and write $C^1_c(I^f_0)$ for the space of compactly supported continuously differentiable functions on $I_0^f$. By \cite[Corollary 2.5]{And75} the form $E^f$ admits a Beurling-Deny decomposition: There are a symmetric nonnegative Radon measure $J^f$ on $I_0^f\times I_0^f\setminus \diag$ and nonnegative Radon measures $\sigma^f$ and $\kappa^f$ on $I_0^f$ such that for all $F\in C_c^1(I_0^f)$ we have 
\begin{equation}\label{E:BD}
E^f(F)=\int F'^2d\sigma^f+\frac12\int\int(F(x)-F(y))^2J^f(dxdy)+\int F^2d\kappa^f.\end{equation}

For any $F,H \in C_c^1(I_0^f)$ we therefore have
\begin{equation}\label{E:imagemeas}
L_{F(f)}(H(f))=2\int H F'^2d\sigma^f+\int H(x)\int(F(x)-F(y))^2J^f(dxdy)+\int H F^2d\kappa^f.
\end{equation}

Now let $\varepsilon>0$ be such that 
\[K:=[-\left\|f\right\|_{\sup}-\varepsilon, \left\|f\right\|_{\sup}+\varepsilon]\subset I^f.\]
The next lemma can be shown in a similar manner as Lemma \ref{L:supports}, therefore we omit its proof. 

\begin{lemma}\label{L:compactsupponedim}
For any $\varepsilon>0$ the compact interval $K$ contains the supports of $\sigma^f$ and $\kappa^f$, and the support of $J^f$ is contained in $K\times K$. 
\end{lemma}

We extend (\ref{E:BD}) and (\ref{E:imagemeas}) to more general functions.

\begin{lemma}\label{L:imagemeas}
Let $K$ be as in the preceding lemma. Formulas (\ref{E:BD}) and (\ref{E:imagemeas}) remain valid for functions $F,H\in C^1(I^f)$ with $F(0)=H(0)=0$ and $F'\geq 0$ on $K$.
\end{lemma}  

\begin{proof} We basically proceed as in Lemma \ref{L:mainstatement}.
By Lemma \ref{L:compactsupponedim} we may assume $F$ has compact support in $I^f$. As before, consider the sequence $(F_n)_n\subset C_c^ 1(I_0^ f)$ defined by (\ref{E:Fn}), i.e.
\[F_n(x):=\int_0^x \varphi_n(t)F'(t)dt-\int_x^0\varphi_n(t)F'(t)dt,\]
where $\varphi_n$ is the continuous function that vanishes in $(-\frac1n,\frac1n)$, equals one outside $(-\frac2n,\frac2n)$ and is linear in between. Then we have $\lim_n F_n'=F'$ and $\lim_n F_n=F$, both monotonically, and consequently 
\[\lim_n \int F_n'^2\:d\sigma^f=\int F'^2\:d\sigma^f\ \ \text{ and }\ \ \lim_n \int F_n^2\:d\kappa^f=\int F'^2\:d\kappa^f\]
and also
\begin{align}
\lim_n \int\int(F_n(x)-F_n(y))^2J^f(dxdy)&=2\lim_n\int\int_{\left\lbrace x<y\right\rbrace}\left(\int_x^y\varphi_n(t)F'(t)dt\right)^2J^f(dxdy)\notag\\
&=\int\int(F(x)-F(y))^2J^f(dxdy).\notag
\end{align}
If $\widetilde{E}^f(F)$ denotes the right-hand side of (\ref{E:BD}) this again shows $\widetilde{E}^f(F)<+\infty$ such that $\lim_n \widetilde{E}^f(F-F_n)=0$ by dominated convergence. Since $\lim_n F_n=F$ uniformly by bounded convergence, the sup-norm-closability of $(\mathcal{E},\mathcal{D})$ yields $E^f(F)=\lim_n E^f(F_n)=\lim_n\widetilde{E}^f(F_n)=\widetilde{E}^f(F)$. If for $H\in C^1(I^f)$ with $H(0)=0$ we now define $\widetilde{L}_{F(f)}(H(f))$ to be the right-hand side of (\ref{E:imagemeas}), then obviously $\widetilde{L}_{F(f)}(H(f))\leq 2\left\|H\right\|_{\sup}\:\widetilde{E}^f(F)$, and using (\ref{E:Lagrangesupbound}), we see that $L_{F(f)}(H(f))=\widetilde{L}_{F(f)}(H(f))$.
\end{proof}

Given a general function $F\in C^1(\mathbb{R})$ that vanishes at zero, let $\widetilde{F}\in C^1(I^f)$ be a compactly supported $C^1$-extension of the restriction of $F$ to $K$ and set $L_{F(f)}:=L_{\widetilde{F}(f)}$. By formula (\ref{E:imagemeas}) and Lemma \ref{L:compactsupponedim} this definition is correct. Moreover, $L_{F(f)}$ admits the integral representation
\[L_{F(f)}(H(f))=\int H\:d\gamma^f(F),\ \ H\in C_c(I_0^f)\]
with the finite nonnegative Radon measure $\gamma^f$, defined by
\[d\gamma^f(F)=F'^2d\sigma^f+\int(F(x)-F(y))^2J^f(dxdy)+F^2d\kappa^f.\]
In particular, we have $\gamma^f(F)(I_0^f)=\mathcal{E}(F(f))$ 
and 
\begin{equation}\label{E:coordidentsingle}
\int H\:d\gamma^f(F)\int_X H(f)\:d\Gamma(F(f)), \ \ H\in C_c(I_0^f),
\end{equation}
where $\Gamma(F(f))$ is the energy measure of $F(f)$ defined in (\ref{E:gammaast}).
By restriction and Lebesgue's theorem identity (\ref{E:coordidentsingle}) remains valid for any bounded Borel function $G$ on $\mathbb{R}$ that vanishes in a neighborhood of zero. If $F$ is constant on a Borel set $A\subset I_0^f$ then
\begin{align}\label{E:nonlocalest}
\int_A\int_{I_0^f} (F(x)-F(y))^2J^f(dxdy)=\int_A\int_{A^c}(F(x)-F(y))^2 J^f(dxdy)\notag\\
\leq \int_{A^c}\int_{I_0^f} (F(x)-F(y))^2 J^f(dxdy)
\end{align}
by Fubini and the symmetry of $J^f$.

Given $\alpha>0$, let $\varphi_\alpha\in C^1(\mathbb{R})$ be a function that satisfies  $\left\|\varphi_\alpha'\right\|_{\sup}\leq 1$ and 
\begin{equation}\label{E:phialpha}
\varphi_\alpha(x)=\begin{cases} 2\alpha \ \  & x\geq 2\alpha\\
x  & -\alpha \leq x\leq \alpha\\
-2\alpha & x\leq -2\alpha.\end{cases}
\end{equation}

\begin{lemma}
For any $f\in\mathcal{D}$ and any $\alpha>0$ we have $\mathcal{E}(f-\varphi_\alpha(f))\leq 4\:\gamma^f(id)((-\alpha,\alpha)^ c)$.
\end{lemma}

\begin{proof}
We have $\mathcal{E}(f-\varphi_\alpha(f))=\gamma^f(id-\varphi_\alpha)(I_0^f)$. Since $id-\varphi_\alpha=0$ on $(-\alpha,\alpha)$, relation (\ref{E:nonlocalest}) implies
\begin{align}
\gamma^f(id-\varphi_\alpha)((-\alpha,\alpha))&\leq \frac12 \int_{(-\alpha,\alpha)^c}\int_{\mathbb{R}} (x-\varphi_\alpha(x)-y+\varphi_\alpha(y))^2J^f(dxdy)\notag\\
&\leq 2\int_{(-\alpha,\alpha)^c}\int_{\mathbb{R}} (x-y)^2 J^f(dxdy).\notag
\end{align}
For $(-\alpha,\alpha)^c$ we use the contractivity of $\varphi_\alpha$ to obtain
\[\gamma^f(id-\varphi_\alpha)((-\alpha,\alpha)^c)\leq 4\sigma^f((-\alpha,\alpha)^c)+2\int_{(-\alpha,\alpha)^c}\int_\mathbb{R} (x-y)^2 J^f(dxdy)+4\int_{(-\alpha,\alpha)^c}x^2\kappa^f(dx).\]
\end{proof}

Together with (\ref{E:coordidentsingle}) the lemma yields the following.
\begin{corollary}\label{C:dominate}
For any $f\in\mathcal{D}$ and any $\alpha>0$ we have $\mathcal{E}(f-\varphi_\alpha(f))\leq 4\:\Gamma(f)(\left\lbrace |f|\geq \alpha\right\rbrace)$.
\end{corollary}

Recall Corollary \ref{C:dominant} and in particular that for any $f\in\mathcal{D}$ we write $\Gamma(f)$ to denote the Radon-Nikodym density $\frac{d\Gamma(f)}{dm} \in L_1(X,m)$. The next lemma shows the uniform $m$-integrability of the densities of $\mathcal{E}$-Cauchy sequences. In a similar form this argument appeared already in \cite[Lemma 2.1]{Schmu92}, we sketch it for completeness.

\begin{lemma}\label{L:UI}
Let $(f_n)_n\subset \mathcal{D}$ be an $\mathcal{E}$-Cauchy sequence. Then the sequence $(\Gamma(f_n))_n$ is convergent in $L_1(X,m)$ and in particular, uniformly $m$-integrable (uniformly absolutely continuous with respect to $m$), i.e. for any $\varepsilon>0$ we can find some $\delta>0$ such that for any set $A\in \sigma(\mathcal{D})$ the relation $m(A)<\delta$ implies
\[\sup_n \Gamma(f_n)(A)<\varepsilon.\]
\end{lemma}
\begin{proof}
Similarly as in \cite[Lemma 2.5(i)]{Hino10} we have
\begin{equation}\label{E:bilinearpointwise}
|\Gamma(f)^{1/2}(x)-\Gamma(g)^{1/2}(x)|\leq \Gamma(f-g)^{1/2}(x)\ \ \text{for $m$-a.a. $x\in X$}
\end{equation}
and for any $f,g\in\mathcal{F}$. This follows easily from an $m$-a.e. valid Cauchy-Schwarz inequality. Reasoning as in \cite[Lemma 2.1]{Schmu92} and \cite[Lemma 2.5(ii)]{Hino10},
\begin{align}
\int_X |\Gamma(f_m)-\Gamma(f_n)|dm & =\int_X |\Gamma(f_m)^{1/2}-\Gamma(f_n)^{1/2}|\left(\Gamma(f_m)^{1/2}+\Gamma(f_n)^{1/2}\right)dm \notag\\
&\leq 2\left(\int_X |\Gamma(f_m)^{1/2}-\Gamma(f_n)^{1/2}|^2dm\right)^{1/2}\:\sup_n\mathcal{E}(f_n)^{1/2},\notag
\end{align}
and integrating (\ref{E:bilinearpointwise}) we see that $\left(\Gamma(f_n)\right)_n\subset L_1(X,m)$ is Cauchy in $L_1(X,m)$, hence convergent, what implies uniform integrability.
\end{proof}

\begin{remark}
Alternatively we could use (\ref{E:contmeasure}) to conclude the uniform integrability from the Vitali-Hahn-Saks Theorem, cf. \cite[Theorem III.7.2]{DS}.
\end{remark}

Thanks to Lemma \ref{L:UI} we can establish a key Proposition which allows to switch from a given $\mathcal{E}$-Cauchy sequence to a sequence that decreases to zero in sup-norm. Recall (\ref{E:phialpha}).

\begin{proposition}\label{P:subseq}
Let $(f_n)_n\subset \mathcal{D}$ be a sequence that is $\mathcal{E}$-Cauchy and converges to zero in $L_2(X,m)$. Then there are a sequence $(k_j)_j\subset \mathbb{N}$ with $\lim_j k_j=\infty$ and a subsequence $(g_j)_j$ of $(f_n)_n$ such that
\begin{equation}\label{E:subCauchy}
\lim_j \mathcal{E}\left(g_j-\varphi_{\frac{1}{k_j}}(g_j)\right)=0.
\end{equation}
\end{proposition}

\begin{proof}
According to Lemma \ref{L:UI} for any $j\in\mathbb{N}\setminus\left\lbrace 0\right\rbrace$ there exists some other $k_j$ such that for any $k\geq k_j$ 
\begin{equation}\label{E:UI}
m(A)<\frac{1}{k}\ \ \text{ implies } \ \ \sup_l\int_A\Gamma(f_{n_l})dm<\frac{1}{j}
\end{equation}
for any $A\in\sigma(\mathcal{D})$.
As $(f_n)_n$ converges to zero in $L_2(X,m)$ we further observe that for any $k$ there is some $n_k$ such that for and $n\geq n_k$ we have
\begin{equation}\label{E:weak}
m\left(|f_n|\geq\frac{1}{k}\right)<\frac{1}{k}.
\end{equation}
Combining (\ref{E:UI}) and (\ref{E:weak}) shows that for any $n\geq n_{k_j}$ we have 
\[\sup_l\int_{\left\lbrace |f_n|\geq\frac{1}{k_j}\right\rbrace}\Gamma(f_{n_l})dm<\frac{1}{j}.\]
Writing $g_j:=f_{n_{k_j}}$ and using Corollary \ref{C:dominate} we see that in particular
\[\mathcal{E}\left(g_j-\varphi_{\frac{1}{k_j}}(g_j)\right)\leq 8\int_{\left\lbrace |g_j|\geq\frac{1}{k_j}\right\rbrace}\Gamma(g_j)dm<\frac{1}{j}.\]
\end{proof}

\begin{remark}
Note that it would be sufficient to require that $(f_n)_n$ converges to zero in $m$-measure. 
\end{remark}

We finally prove Theorem \ref{T:clos}.

\begin{proof}
Let $(f_n)_n\subset\mathcal{D}$ be a sequence that is $\mathcal{E}$-Cauchy and converges to zero in $L_2(X,m)$. Let $(g_j)_j$ be the subsequence of $(f_n)_n$ and $(k_j)_j$ the corresponding sequence of indices with (\ref{E:subCauchy}), shown to exist in Proposition \ref{P:subseq}. Clearly $(g_j)_j$ is $\mathcal{E}$-Cauchy, too. We have
\[\mathcal{E}\left(\varphi_{\frac{1}{k_j}}(g_j)-\varphi_{\frac{1}{k_l}}(g_l)\right)^{1/2}\leq \mathcal{E}\left(\varphi_{\frac{1}{k_j}}(g_j)-g_j\right)^{1/2}+\mathcal{E}\left(\varphi_{\frac{1}{k_l}}(g_l)-g_l\right)^{1/2}+\mathcal{E}(g_j-g_l)^{1/2},\]
which by (\ref{E:subCauchy}) is arbitrarily small, provided $j$ and $l$ are large enough. Consequently 
\[u_j:=\varphi_{\frac{1}{k_j}}(g_j)\]
defines an $\mathcal{E}$-Cauchy sequence $(u_j)_j$. By construction
\[\sup_{x\in X}|u_j(x)|\leq \frac{2}{k_j}\]
for all $j$, and sup-closability implies $\lim_j \mathcal{E}(u_j)=0$. Another application of (\ref{E:subCauchy}) yields $\lim_j\mathcal{E}(g_j)=0$, and since $(g_j)_j$ is an $\mathcal{E}$-convergent subsequence of the $\mathcal{E}$-Cauchy sequence $(f_n)_n$, 
\[\lim_n \mathcal{E}(f_n)=0.\]
The bilinear form $(\mathcal{E},\mathcal{D})$ is a closable, densely defined and positive definite symmetric bilinear form on $L_2(X,m)$ on which normal contractions operate. Hence its closure is a Dirichlet form. 
\end{proof}

Corollary \ref{C:domandclos} now is an immediate consequence of Theorem \ref{T:clos} together with Remark \ref{R:lcccase} (i). Another modification of Theorem \ref{T:clos} for Lagrangians reads as follows. 

\begin{theorem}\label{T:clos2}
Let $\mathcal{D}$ be a $C^1_c(\mathbb{R}^2)$-stable algebra containing a strictly positive function $\chi>0$ and let $(L,\mathcal{D})$ be a sup-closable Lagrangian on which normal contractions operate. If $m$ is an energy dominant measure for $(L,\mathcal{D})$, then its energy $(\mathcal{E}_L,\mathcal{D})$ is closable on $L_2(X,m)$, and its closure is a Dirichlet form that admits a carr\'e du champ.
\end{theorem}

As yet another consequence of Theorem \ref{T:clos} together with \cite[Corollary I.3.3.2]{BH91} we obtain a Leibniz type bound.

\begin{corollary} Let $\mathcal{D}$ be a $C^1_c(\mathbb{R}^2)$-stable algebra containing a strictly positive function $\chi>0$, let $(\mathcal{E},\mathcal{D})$ be a sup-norm-closable bilinear form on which normal contractions operate and let $m$ be an energy dominant measure. Then we have
\begin{equation}\label{E:Leibnizest}
\mathcal{E}(fg)^{1/2}\leq \mathcal{E}(f)^{1/2}\left\|g\right\|_{\sup}+\mathcal{E}(g)^{1/2}\left\|f\right\|_{\sup},\ \ f,g\in\mathcal{D}.
\end{equation}
\end{corollary}

For Dirichlet forms Theorems \ref{T:Dirichletform} and \ref{T:clos} yield a change of measure result. 

\begin{corollary}\label{C:changeofmeasure}
Let $(X,\mathcal{X},\mu)$ be a $\sigma$-finite measure space, let $(\mathcal{E},\mathcal{F})$ be a Dirichlet form on $L_2(X,\mathcal{X},\mu)$ and let $\mathcal{B}$ be defined as in (\ref{E:B}). Assume there exists a bounded and strictly positive $\mathcal{X}$-measurable function $\chi>0$ with $\mu$-class in $\mathcal{F}$. Then for any energy dominant measure $m$ the form $(\mathcal{E},\mathcal{B})$ is closable in $L_2(X,\sigma(\mathcal{B}), m)$ and its closure is a Dirichlet form that admits a carr\'e du champ. If in addition the measure space $(X,\sigma(\mathcal{B}),\mu)$ is complete, then we have $\sigma(\mathcal{B})=\mathcal{X}$.
\end{corollary}

Note that the existence of $\chi>0$ as in the corollary forces the $\sigma$-finiteness of $\mu$ by (\ref{E:chik}). We only use the redundant formulation to comply with the definition of Dirichlet forms as in \cite{BH91}.

\begin{proof}
We comment only on the last sentence, everything else is immediate from Theorem \ref{T:clos} and Corollary \ref{C:Dirichletform}. The inclusion $\sigma(\mathcal{B})\subset \mathcal{X}$ is trivial. To prove the converse, recall that by (\ref{E:chik}) there is a sequence of sets $A_k\in \mathcal{X}\cap\sigma(\mathcal{B})$ with $X=\bigcup_{k=1}^\infty A_k$ and $\mu(A_k)<+\infty$, $A_k\subset A_{k+1}$, $k\geq 1$. Clearly $\mathcal{B}$ is a dense subset of $\mathcal{F}$ (in the sense of equivalence classes) and therefore also of $L_2(X,\mathcal{X},\mu)$. Let $A\in\mathcal{X}$. For any $k$ we can find a sequence $(f_n^{(k)})_n\subset \mathcal{B}^+$ converging to $\mathbf{1}_{A\cap A_k}$ in $L_2(X,\mathcal{X},\mu)$ and having a subsequence $(f_{n_l}^{(k)})_l$ such that
\[\lim_l f_{n_l}^{(k)}=\mathbf{1}_{A\cap A_k}\ \ \text{$\mu$-a.e. on $A_k$}.\]
Since $\mathbf{1}_{A_k}f_n\in\sigma(\mathcal{B})$ and $(X,\sigma(\mathcal{B}),\mu)$ is complete, the function $\mathbf{1}_{A\cap A_k}$ is $\sigma(\mathcal{B})$-measurable. As
$\lim_{k}\mathbf{1}_{A\cap A_k}=\mathbf{1}_A$ we obtain $A\in\sigma(\mathcal{B})$. 
\end{proof}

\begin{remark}\mbox{} 
\begin{enumerate}
\item[(i)] If $\mathcal{X}$ is generated by a countable semiring and there exists an energy dominant measure for $(\mathcal{E},\mathcal{B})$ then the measure in (\ref{E:standardsum}) is energy dominant, cf. Remark \ref{R:energysep} and Theorem \ref{T:energydom}.
\item[(ii)] We have $\sigma(\mathcal{B})=\mathcal{X}$ if $\mathcal{X}$ is generated by a countable collection $\left\lbrace f_n\right\rbrace_n$ of bounded real valued functions with $\mathcal{E}(f_n)<+\infty$ for all $n$. 
\item[(iii)] If $X$ is a locally compact separable metric space, $(\mathcal{E},\mathcal{F})$ is a regular Dirichlet form and $\mathcal{D}$ contains a core of $(\mathcal{E},\mathcal{F})$, then $\sigma(\mathcal{D})$ will again be the Borel $\sigma$-algebra over $X$.
\item[(iv)] Note that Corollary \ref{C:changeofmeasure} does not need a topology on $X$. However, it is related to well known change of measure results for regular Dirichlet forms: For a regular Dirichlet form $(\mathcal{E},\mathcal{F})$ on a locally compact separable metric space $X$ as in Remark \ref{R:lcccase} the statement of Corollary \ref{C:changeofmeasure} can be shown either by using the Beurling-Deny representation of $(\mathcal{E},\mathcal{F})$ or by considering the $\mu$-symmetric Hunt process on $X$ uniquely associated with $(\mathcal{E},\mathcal{F})$. The change of reference measure from $\mu$ to $m$ corresponds to a time change for this process. The measure $m$ defined as in (\ref{E:standardsum}) does not charge sets of zero capacity but has full quasi-support, at least if $(\mathcal{E},\mathcal{F})$ is irreducible or transient. We may then perform the change of measure by the probabilistic arguments given in \cite[Section 5, in particular Theorem 5.3]{FLJ} and \cite{FST}. See also \cite{KN91}, \cite[
Section 6.2, p. 275]{FOT94}, \cite[Corollary 5.2.10]{ChF12} and \cite[Section 5]{HRT}.
\end{enumerate}
\end{remark}

\section{Bilinear forms and Lagrangians on the Gelfand spectrum}\label{S:Gelfand}

As before, let $\mathcal{D}$ be an algebra of bounded real valued functions on a nonempty set $X$, endowed with the supremum norm. For many bilinear forms $(\mathcal{E},\mathcal{D})$ energy dominant measures on the space $X$ might not exist or we may just not be able to prove their existence. However, we can always transfer $(\mathcal{E},\mathcal{D})$ to a 'compactification' of $X$, and for the transferred form energy measures do exist. 

Let $\mathcal{A}(\mathcal{D})$ denote the commutative $C^\ast$-algebra generated by the natural complexification of $\mathcal{D}$ and write $\Delta$ to denote the \emph{Gelfand spectrum} of $\mathcal{A}(\mathcal{D})$, that is the space of nonzero complex valued multiplicative linear functionals on $\mathcal{A}(\mathcal{D})$. The Gelfand spectrum $\Delta$ is a locally compact Hausdorff space. It is second countable whenever $\mathcal{D}$ is countably generated. For any $f\in\mathcal{A}(\mathcal{D})$ the Gelfand transform $\hat{f}:\Delta\to\mathbb{C}$ of $f$ is defined by $\hat{f}(\varphi):=\varphi(f)$, $\varphi\in\Delta$. According to the Gelfand representation theorem the map $f\mapsto \hat{f}$ defines an $^\ast$-isomorphism from $\mathcal{A}(\mathcal{D})$ onto the algebra $C_{\mathbb{C},0}(\Delta)$ of complex valued continuous functions on $\Delta$ that vanish at infinity. If $\mathcal{D}$ vanishes nowhere on $X$ then the image $\iota(X)$ of $X$ under the evaluation map $\iota:X\to\Delta$, $\iota(x)(f):=f(x)
$, is dense in $\Delta$. The algebra $\hat{\mathcal{D}}:=\left\lbrace \hat{f}\in C(\Delta):f\in\mathcal{D}\right\rbrace$ is uniformly dense in the subalgebra $C_0(\Delta)$ of real valued continuous functions on $\Delta$ vanishing at infinity. A positive linear functional $L:\mathcal{D}\to\mathbb{R}$ naturally extends to a positive linear functional $\hat{L}$ on $\hat{D}$ by $\hat{L}(\hat{f}):=L(f)$ and by boundedness further to $C_0(\Delta)$. If $(\mathcal{E},\mathcal{D})$ a given bilinear form then
\[\hat{\mathcal{E}}(\hat{f},\hat{g}):=\mathcal{E}(f,g), \ \ \hat{f},\hat{g}\in\hat{\mathcal{D}}\]
defines a bilinear form $\hat{\mathcal{E}}$ on $\hat{\mathcal{D}}$. To $(\hat{\mathcal{E}},\hat{\mathcal{D}})$ we refer as the \emph{transferred form}. If $f\in\mathcal{D}$ and $F$ is a normal contraction, then there exists a sequence $(p_n)_n$ of polynomials such that
$F(\hat{f})=\lim_n p_n(f)^\wedge =\lim_n (p_n(f))^\wedge =(F(f))^\wedge \in \hat{\mathcal{D}}$,
the limits being uniform. This also implies $\hat{\mathcal{E}}(F(\hat{f}))=\hat{\mathcal{E}}((F(f))^\wedge)=\mathcal{E}(F(f))\leq \mathcal{E}(f)=\hat{\mathcal{E}}(\hat{f})$,
i.e. normal contractions operate on $(\hat{\mathcal{E}},\hat{\mathcal{D}})$. Furthermore, it is immediate to see that $(\hat{\mathcal{E}},\hat{\mathcal{D}})$ is sup-norm-closable if and only if $(\mathcal{E},\mathcal{D})$ is. Similar statements holds for energy separability and $C^1(\mathbb{R}^2)$-stability.

Now the following is an immediate consequence of Remark \ref{R:lcccase} (i).

\begin{theorem}\label{T:closGelfand}
Let $\mathcal{D}$ be a $C^1_c(\mathbb{R}^2)$-stable algebra containing a strictly positive function $\chi>0$ and let $(\mathcal{E},\mathcal{D})$ be an energy separable bilinear form on which normal contractions operate. Then there exists an energy dominant (Radon) measure $\hat{m}$ for the transferred form $(\hat{\mathcal{E}},\hat{\mathcal{D}})$. If in addition $(\mathcal{E},\mathcal{D})$ is sup-norm-closable, then $(\hat{\mathcal{E}},\hat{\mathcal{D}})$ extends to a Dirichlet form on $L_2(\Delta,\hat{m})$ that admits a carr\'e du champ. The space $\hat{\mathcal{D}}$ is dense in $\hat{\mathcal{F}}$ and dense in the space $C_c(\Delta)$ of continuous compactly supported functions on $\Delta$.
\end{theorem}

Together with Theorem \ref{T:Dirichletform} and Proposition \ref{P:LtoEL} we now obtain a proof of Corollary \ref{C:ELtoL}.

\begin{proof}
By Theorem \ref{T:Dirichletform} the Lagrangian $(\hat{L},\hat{\mathcal{D}})$, given by $L_{\hat{f}}(\hat{h})=2\hat{\mathcal{E}}(\hat{f}\hat{h},\hat{f})-\hat{\mathcal{E}}(\hat{f}^2,\hat{h})$, $\hat{f},\hat{h}\in\hat{\mathcal{D}}$,
is sup-norm-closable. By the properties of the Gelfand transform we have $L_f(h)=L_{\hat{f}}(\hat{h})$, $f,h\in\mathcal{D}$, for the Lagrangian $(L,\mathcal{D})$ generated by $(\mathcal{E},\mathcal{D})$.
\end{proof}

\begin{remark}\label{R:pullback}
\item[(i)] In general it may not be possible to pull the measure $\hat{m}$ back to a measure on the space $X$, see \cite{HKT} for a counterexample.
\item[(ii)] Assume that $\mathcal{D}$ vanishes nowhere on $X$ and separates the points of $X$. Then the map $\iota$ is injective and $X$ is (densely) embedded in $\Delta$ as $\iota(X)$. If $\hat{d}$ is a metric on $\mathcal{D}$, then its pull-back $d(x,y):=\hat{d}(\iota(x),\iota(y))$, $x,y\in X$, defines a metric on $X$. One interesting idea may be to consider the situation when  $\Delta$ is metrizable (for instance if $\mathcal{D}$ is countably generated and $\mathbf{1}\in \mathcal{D}$) and $\hat{d}$ metrizes the Gelfand topology. Another interesting direction might be to consider intrinsic metrics on $\Delta$ in the pointwise sense, see e.g. \cite{Sturm94, Sturm95} or in the measurable sense, see \cite{HinoRam}. Roughly speaking \cite[Theorem 1.2]{HinoRam} implies that in the (strongly) local case (cf. \cite[Section I.5]{BH91}) and with respect to any finite energy dominant measure and corresponding intrinsic metric in the measurable sense the transferred form $(\hat{\mathcal{E}},\hat{\mathcal{F}})$ 
admits Gaussian short time asymptotics on $\Delta$.
\end{remark}

\section{Sup-norm-lower semicontinuity}\label{S:lsc}

In this section we aim to compare sup-norm-closability and sup-norm-lower semicontinuity of bilinear forms.  

\begin{definition}
Let $\mathcal{D}$ be a space of bounded real valued functions on a nonempty set $X$ endowed with the supremum norm and denote its completion by $\overline{\mathcal{D}}$. Let $\mathcal{E}:\overline{\mathcal{D}}\to [0,+\infty]$ be a quadratic (extended real valued) functional such that $\mathcal{E}(f)<+\infty$ for any $f\in\mathcal{D}$. We say that $\mathcal{E}$ is \emph{sup-norm-lower semicontinuous} if for any sequence $(f_n)_n\subset \overline{\mathcal{D}}$ with uniform limit $\lim_n f_n=f$ we have $\mathcal{E}(f)\leq \liminf_n \mathcal{E}(f_n)$.
\end{definition}

If $\mathcal{D}$, $\overline{\mathcal{D}}$ and $\mathcal{E}$ are as in the definition, then polarization yields a bilinear form $(\mathcal{E},\mathcal{D})$ in the sense of the preceding sections. 

\begin{lemma}\label{L:suplsctosupclos}
If $\mathcal{E}$ is sup-norm-lower semicontinuous then $(\mathcal{E},\mathcal{D})$ is sup-norm-closable.
\end{lemma}

\begin{proof}
Let $(f_n)_n\subset\mathcal{D}$ be an $\mathcal{E}$-Cauchy sequence such that $\lim_n f_n=0$ uniformly. Then for any $n$ we have $\lim_m f_n-f_m=f_n$, hence $\mathcal{E}(f_n)\leq\liminf_m\mathcal{E}(f_n-f_m)$. Given $\varepsilon>0$ therefore $\mathcal{E}(f_n)<\varepsilon$ whenever $n$ is sufficiently large.
\end{proof}

Under additional assumptions we obtain the following result.

\begin{theorem}
Let $\mathcal{D}$ be an algebra containing a strictly positive function $\chi>0$ and let $\overline{\mathcal{D}}$ denote its completion in uniform norm. Let $\mathcal{E}$ be an nonnegative extended real valued quadratic functional on $\overline{\mathcal{D}}$, finite on $\mathcal{D}$ and such that normal contractions operate on $(\mathcal{E},\mathcal{D})$. If $m$ is a finite energy dominant measure for $(\mathcal{E},\mathcal{D})$ then the following are equivalent:
\begin{enumerate}
\item[(i)] The form $(\mathcal{E},\mathcal{D})$ is sup-norm closable.
\item[(ii)] The form $(\mathcal{E},\mathcal{D})$ is closable in $L_2(X,m)$.
\item[(iii)] The functional $\mathcal{E}$ is lower semicontinuous in the $L_2(X,m)$-sense.
\item[(iv)] The functional $\mathcal{E}$ is sup-norm-lower semicontinuous.
\end{enumerate}
\end{theorem}

\begin{proof}
The equivalence of (i) and (ii) is already established by Theorems \ref{T:clos} and \ref{T:Dirichletform}. By Lemma \ref{L:suplsctosupclos} (iv) implies (i). Because $m$ is finite, (iv) follows from (iii). The remaining implication from (ii) to (iii) is just the standard statement that Dirichlet forms are lower semicontinuous. 
\end{proof}

\begin{remark}
In the locally compact and separable case the equivalence of (ii), (iii) and (iv) already follows from \cite{Mo95}. 
\end{remark}

\section*{Appendix}

We sketch a proof of Theorem \ref{T:positive2}. Our aim is to apply the method of Andersson \cite{And75}, and to do so we equip the coordinate forms $E^{f,g}$ with a somewhat artificial but suitable domain. Given $f\in\mathcal{D}$ let $I^f\subset\mathbb{R}$ be an bounded open interval such that $[-\left\|f\right\|_{\sup},\left\|f\right\|_{\sup}]\subset I^f$.

First consider $\lip(I^f)\otimes \lip(I^g)$, any member of this space may be viewed as a finite linear combination $(x_1,x_2)\mapsto \sum_i \varphi_i^1(x_1)\varphi_i^2(x_2)$ with $\varphi_i^1\in \lip(I^f)$ and $\varphi_i^2\in \lip(I^g)$. The linear extension of the definition \[(\varphi^1\otimes\varphi^2)(\psi^1\otimes\psi^2):=(\varphi^1\psi^1)\otimes (\varphi^2\psi^2):=(\varphi^1\psi^1)\otimes (\varphi^2\psi^2)\]
turns $\lip(I^f)\otimes \lip(I^g)$ into an algebra, and by $\mathcal{R}_0(\lip(I^f)\otimes \lip(I^g))$ we denote the algebra generated by all bounded functions of type $p(\varphi)(q(\psi))^{-1}$, where $\varphi,\psi\in \lip(I^f)\otimes \lip(I^g)$ with $\varphi(0)=\psi(0)=0$ and $p$ and $q$ are polynomials with $p=0$ and $\left\lbrace q=0\right\rbrace\subset \left\lbrace p=0\right\rbrace$. Due to the invertibility condition in Theorem \ref{T:positive2} this algebra is contained in $\mathcal{D}$. Finally, let $\mathcal{V}^{f,g}\subset \mathcal{D}$ denote the space of all its Lipschitz transformations,
\[\mathcal{V}^{f,g}:=\left\lbrace F(\Phi):\ \Phi\in \mathcal{R}_0(\lip(I^f)\otimes \lip(I^g)), \ \ F\in \lip(\mathbb{R}),\ F(0)=0\right\rbrace.\]
By Stone-Weierstrass $\mathcal{V}^{f,g}$ is a uniformly dense subalgebra of $C_0(I^{f,g}_0)$.
Similarly as before consider
\[E^{f,g}(F,G):=\mathcal{E}(F(f,g),G(f,g)),\ \ F,G\in \mathcal{V}^{f,g},\]
which yields a symmetric and nonnegative definite bilinear on which normal contractions operate. 
Therefore we have the Beurling-Deny decomposition (\ref{E:BD}), but some additional arguments are needed to arrive at the measure representation (\ref{E:measurerep}) for the strongly local part. Slight modifications of \cite[Theorems 2.2 and 2.3]{And75} yield the following estimate. 

\begin{lemma}\label{L:aprioribound}
For any compact set $K\subset I_0^{f,g}$ there exists a constant $C_K>0$ such that 
\begin{equation}\label{E:aprioribound}
N^{f,g}(F,G)\leq C_K\:\left\|\nabla F\right\|_{\sup}\left\|\nabla G\right\|_{\sup}
\end{equation}
whenever $F,G\in \mathcal{V}^{f,g}\cap  C_c^1(I^{f,g}_0)$ are such that $\supp F, \supp G\subset K$.
\end{lemma}

\begin{proof}
We can follow Andersson \cite[Section 3]{And75}, the only news being that we choose cut off functions from $\mathcal{V}^{f,g}\cap  C_c^1(I^{f,g}_0)$. For any open neighborhood $U_x\subset I_0^{f,g}$ of a point $x=(x_1,x_2)\in I^{f,g}_0$ we can find open intervals $J^f\subset I^f$ and $J^g\subset I^g$ conatining $x_1$ and $x_2$, respectively, and such that $\overline{J^{f}}\times\overline{J^g}\subset U_x$. Consequently there are also $C^1$-functions $\varphi_1$ and $\varphi_2$ with $0\leq \varphi_i\leq 1$, $\varphi_1\equiv 1$ on $\overline{J^f}$ and $\varphi_2\equiv 1$ on $\overline{J^g}$ such that $\psi:=\varphi_1\otimes\varphi_2\in \mathcal{V}^{f,g}\cap C^1(I_0^{f,g})$ is supported inside $U_x$. 

Now consider the set
\begin{multline}
E:=\left\lbrace x\in I_0^{f,g}: \text{ for any open neighborhood $U_x\subset I_0^{f,g}$ of $x$ there exists $\psi\in  \mathcal{V}^{f,g}\cap C^1(I_0^{f,g})$}\right.\notag\\
\left. \text{with $\supp \psi\subset U_x$ and $\left\|\nabla \psi\right\|_{\sup}\leq 1$ such that $N^{f,g}(\psi)\geq 1$}\right\rbrace.
\end{multline} 
as in \cite[Lemma 3.1]{And75} we can see that $E$ is locally finite. Next, if $K\subset I_0^{f,g}\setminus E$ is compact then there exists a constant $C_K>0$ such that for any $F,G\in \mathcal{V}^{f,g}\cap C^1(I_0^{f,g})$ with $\supp F, \supp G\subset K$ we have 
\begin{equation}\label{E:Nest}
|N^{f,g}(F,G)|\leq C_K\left\|\nabla F\right\|_{\sup}\left\|\nabla G\right\|_{\sup}.
\end{equation}
To see (\ref{E:Nest}) note that for any $x\in K$ there exists a neighborhood $U_x$ such that for any $\psi \in \mathcal{V}^{f,g}\cap C^1(I_0^{f,g})$ with $\supp\psi\subset U_x$ we have 
\[N^ {f,g}(\psi)\leq \left\|\nabla \psi\right\|_{\sup}^2.\]
Due to compactness we can find $x^1,\dots, x^n\in K$ with neighborhoods $U_{x^i}$ that provide a finite open cover of $K$. Moreover, by the construction of $\mathcal{V}^{f,g}$, we can furnish a corresponding partition of unity, i.e. find functions $\chi_1,\dots, \chi_n$ such that $\chi_k\in \mathcal{V}^{f,g}\cap C^1(I_0^{f,g})$, $0\leq \chi_i\leq 1$ and $\supp\chi_i\subset U_{x^i}$, $\sum_{i=1}^n \chi_i=1$ on $K$. This allows to proceed further along the arguments of \cite[p. 16]{And75} to obtain (\ref{E:Nest}), which correspondes to \cite[Lemma 3.2]{And75}. Likewise, it is not difficult to prove analogs of \cite[Lemmas 3.3 and 3.4]{And75}. To obtain an analog of \cite[Lemma 3.5]{And75}, consider
\[\varphi(r):=\sup_{|x|\leq r}\max\left\lbrace |\frac{\partial f}{\partial x_1}(x)|,|\frac{\partial f}{\partial x_2}(x)|\right\rbrace, \ \ r>0,\]
for a fixed function $F\in \mathcal{V}^{f,g}\cap C^1(I_0^{f,g})$ vanishing at the origin
and note that close to the origin it is dominated by
\[\int_0^{|x_1|}\varphi(r)dr+\int_0^{|x_2|}\varphi(r)dr.\]
Now Lemma \ref{L:aprioribound} follows as in the proof of \cite[Theorem 2.3]{And75}.
\end{proof}

Lemma \ref{L:aprioribound} allows a polynomial approximation argument.

\begin{corollary}\label{C:polyapprox}
For any $F\in C_c^1(I^{f,g}_0)$ there exist a sequence $(F_n)_n\subset \mathcal{V}^{f,g}\cap C_c^1(I^{f,g}_0)$ and a compact set $K\subset I^{f,g}_0$ with $\supp F\subset K$ and $\supp F_n\subset K$ for all $n$ such that $\lim_n F_n=F$ uniformly and $\lim_{n\to\infty}\left\|\nabla F_n-\nabla F\right\|_{\sup}=0$.
\end{corollary}

\begin{proof} Let $K$ be the closure $\overline{U}$ of a bounded open set $U\subset I_0^{f,g}$ containing $\supp F$ and let $0\leq \theta\leq 1$ be a function from $\mathcal{V}^{f,g}\cap C_c^1(I^{f,g}_0)$ that equals one on $\supp F$ and is supported inside $U$. Without loss of generality we may assume that $K=[0,1]^2$. Given $k,l\geq 1$ consider the Bernstein polynomials
\[F_{k,l}(x_1,x_2):=\sum_{p=0}^k\sum_{q=0}^l F(\frac{p}{k},\frac{q}{l})\binom{p}{k}\binom{q}{l}x_1^p(1-x_1)^{k-p}x_2^q(1-x_2)^{l-q}.\]
Given any $n$ we have $\sup_{x\in K}|F_{k,l}(x)-F(x)|<1/n$ and $\sup_{x\in K}|\nabla F_ {k,l}(x)-\nabla F(x)|<1/k$ for any $k$ and $l$ both greater than or equal to some number $N_n$, see \cite{Ki51}. Setting $F_n:=\theta F_{N_n,N_n}$ yields a sequence $(F_n)_n \subset \mathcal{V}^{f,g}\cap C_c^1(I^{f,g}_0)$ that converges uniformly to $F$ on $I_0^{f,g}$ and with gradients $\nabla F_n$ converging uniformly to $\nabla F$. 
\end{proof}

Corollary \ref{C:polyapprox} now allows to extend $E^{f,g}$ to a bilinear form on $C_c^1(I_0^{f,g})$. For $F\in C_c^1(I_0^{f,g})$ let $(F_n)_n$ be a sequence as in Corollary \ref{C:polyapprox} and set
\begin{equation}\label{E:extendEfg}
E^{f,g}(F):=\lim_n E^{f,g}(F_n).
\end{equation}
Corollary \ref{C:polyapprox} together with the sup-norm-closability of $(\mathcal{E},\mathcal{D})$ and formula (\ref{E:BD2}) imply the following.

\begin{corollary}\label{C:extendEfg}
Definition \ref{E:extendEfg} is correct and yields a sup-norm closable symmetric non-negative definite bilinear form on $C_c^1(I_0^{f,g})\cup \mathcal{V}^{f,g}$ on which normal contractions operate. Formulas (\ref{E:BD2}) and (\ref{E:aprioribound}) remain valid for  all $F,G\in C_c^1(I^{f,g}_0)$.
\end{corollary}

For this extension we have (\ref{E:measurerep}) and therefore may proceed as in the the proof of Theorem \ref{T:positive}.

\end{document}